\documentclass[letterpaper, 11pt]{article}
\usepackage[utf8]{inputenc}
\usepackage[T1]{fontenc}
\usepackage{lmodern}
\usepackage{graphicx,amsmath,amsfonts,latexsym,amsthm,amssymb,mathrsfs}
\usepackage[abs]{overpic}
\usepackage[usenames,dvipsnames]{xcolor}
\usepackage[linktocpage=true]{hyperref}
\usepackage{caption}
\usepackage{dsfont}
\usepackage{mathtools}
\usepackage{stmaryrd}
\usepackage{tikz-cd}
\usepackage{hyperref}
\usepackage{graphpap}
\usepackage{svg}
\tikzset{
  symbol/.style={
    draw=none,
    every to/.append style={
      edge node={node [sloped, allow upside down, auto=false]{$#1$}}}
  }
}
\usepackage{comment}
\usepackage{enumitem}
\numberwithin{equation}{section}
\usepackage{titlesec}
\usepackage{microtype}

\usepackage[margin=1.3in]{geometry}

\usepackage[backend=biber, maxnames=4, style=alphabetic, sorting=nyt]{biblatex}
\hypersetup{
  colorlinks,
  citecolor=teal,
  linkcolor=purple,
  urlcolor=teal}

\definecolor{green1}{rgb}{0.,0.4,0.}
\definecolor{blue1}{rgb}{0.,0.,0.8}
\definecolor{red1}{rgb}{0.8,0.,0.}
  
\newtheorem{thmintro}{Theorem}
\newtheorem{corintro}{Corollary}
\newtheorem{question}{Question}

\newtheorem{thm}{Theorem}[section]
\newtheorem{cor}[thm]{Corollary}
\newtheorem{lem}[thm]{Lemma}
\newtheorem{prop}[thm]{Proposition}
\newtheorem{defn}[thm]{Definition}

\newtheorem{rem}[thm]{Remark}

\newtheorem*{question*}{Questions}
\newtheorem*{thm*}{Theorem}
\newtheorem*{cor*}{Corollary}
\newtheorem*{warn}{Warning}

\newcommand{\N}{\mathbb{N}}

\newcommand{\R}{\mathbb{R}}

\title{Anosov flows and Liouville pairs in dimension three}

\author{Thomas Massoni\thanks{Department of Mathematics, Princeton University, USA. Email address: \href{mailto:tmassoni@princeton.edu}{tmassoni@princeton.edu}}}

\date{\today}

\begin{document}

\maketitle


\begin{abstract}
Building upon the work of Mitsumatsu and Hozoori, we establish a complete homotopy correspondence between three-dimensional Anosov flows and certain pairs of contact forms that we call \emph{Anosov Liouville pairs}. We show a similar correspondence between projectively Anosov flows and bi-contact structures, extending the work of Mitsumatsu and Eliashberg--Thurston. As a consequence, every Anosov flow on a closed oriented three-manifold $M$ gives rise to a Liouville structure on $\R \times M$ which is well-defined up to homotopy, and which only depends on the homotopy class of the Anosov flow. Our results also provide a new perspective on the classification problem of Anosov flows in dimension three.
\end{abstract}


\maketitle

\tableofcontents


\section{Introduction}

Throughout this article, $M$ denotes a closed, oriented, smooth manifold of dimension three. We will always assume that the Anosov and projectively Anosov flows on $M$ under consideration are \textbf{oriented}, i.e., their stable and unstable foliations are oriented. This can always be achieved by passing to a suitable double cover of $M$. For simplicity, we will only consider smooth (i.e., $\mathcal{C}^\infty$) flows, as we are primarily interested in smooth contact and symplectic structures. Our main results hold for (projectively) Anosov flows generated by $\mathcal{C}^1$ vector fields with minor changes. Moreover, the structural stability of $\mathcal{C}^1$ Anosov vector fields~\cite{R75} ensures that any Anosov flow generated by a $\mathcal{C}^1$ vector field is topologically equivalent to a smooth Anosov flow, and these two flows are dynamically identical. The definitions and basic properties of Anosov and projectively Anosov flows are recalled in Section~\ref{sec:anosov}.

\medskip

The notion of Anosov flow, originally introduced by Anosov~\cite{A63, A67} as a generalization of the geodesic flow on hyperbolic manifold, plays a central role in the theory of smooth dynamical systems. The interplay between the dynamical and topological properties of Anosov flows is particularly rich and striking in dimension three. We refer to the nice survey~\cite{B17} for many relevant results and references, and to the book~\cite{FH19} for a more complete exposition. Eliashberg and Thurston~\cite{ET}, and independently Mitsumatsu~\cite{M95}, introduced the more general concept of a \emph{conformally/projectively Anosov flow} on three-manifolds, and established a correspondence between such flows and \emph{bi-contact structures}, i.e., transverse pairs of contact structures with opposite orientations. Recently, Hozoori~\cite{H22a} extended this correspondence to Anosov flows, and showed that (oriented) Anosov flows can be completely characterized in terms of bi-contact structures admitting a pair of contact forms satisfying a natural symplectic condition. More precisely, Hozoori showed the following

\begin{thm*}[\cite{H22a}, Theorem 1.1]
Let $\Phi$ be a non-singular flow on a closed oriented $3$-manifold $M$, generated by a vector field $X$. $\Phi$ is oriented Anosov if and only if there exist transverse contact structures $\xi_-$ and $\xi_+$, negative and positive, respectively, and contact forms $\alpha_-$ and $\alpha_+$ for $\xi_-$ and $\xi_+$, respectively, such that the $1$-forms $$ (1-t) \alpha_- + (1+t) \alpha_+ \qquad \textrm{and} \qquad  - (1-t) \alpha_- + (1+t) \alpha_+$$ are positively oriented Liouville forms on $[-1,1]_t \times M$.
\end{thm*}

Recall that a Liouville form on a smooth manifold with boundary $V$ is a $1$-form $\lambda$ such that $\omega = d \lambda$ is symplectic, i.e., non-degenerate, and the Liouville vector field $Z$ defined by $\omega(Z, \cdot \, ) = \lambda$ is outward-pointing along the boundary of $V$. The pair $(V, \lambda)$ is called a \textbf{Liouville domain}. The above theorem shows in particular that an Anosov flow on a $3$-manifold $M$ (under some suitable orientability assumptions recalled in Definition~\ref{definition:defanosov}) gives rise to a Liouville structure on $[-1,1] \times M$ which is \emph{not Weinstein}, since the latter manifold has a non-trivial third homology group and disconnected boundary. It is natural to ask the following

\begin{question*} How do the Liouville structures constructed by Hozoori depend on the underlying Anosov flow? More precisely,
	\begin{enumerate}[label=(\arabic*)]
	\item For a given Anosov flow $\Phi$, is the space of pairs of contact forms $(\alpha_-, \alpha_+)$ as in the previous Theorem path-connected? 
	\item Does a path of Anosov flows induce a path of Liouville structures on $[-1,1] \times M$?
	\item Does every bi-contact structure $(\xi_-, \xi_+)$ supporting an Anosov flow admit a pair of contact forms $(\alpha_-, \alpha_+)$ as in the previous Theorem?
	\end{enumerate}
\end{question*}

Here, we say that a bi-contact structure $(\xi_-, \xi_+)$ supports a non-singular flow generated by a vector field $X$ if $X \in \xi_- \cap \xi_+$ (in the more precise Definition~\ref{def:bicont}, we also add a condition on the orientations of $\xi_\pm$).

\medskip

In the present article, we give a complete answer to these questions and upgrade Hozoori's correspondence to a \emph{homotopy equivalence} between the space of Anosov flows on $M$, and a space of suitable pairs of contact forms on $M$. To that extent, we will consider a \emph{different} condition on the pair $(\alpha_-, \alpha_+)$ than the one in Hozoori's theorem, and we first show

\begin{thmintro} \label{thmintro:supporting}
Let $\Phi$ be a non-singular flow on a closed oriented $3$-manifold $M$, generated by a  vector field $X$. $\Phi$ is oriented Anosov if and only if there exists a pair of contact forms $(\alpha_-, \alpha_+)$ on $M$ such that $X \in \ker \alpha_- \cap \ker \alpha_+$, and the $1$-forms $$ e^{-s} \alpha_- + e^s \alpha_+ \qquad \textrm{and} \qquad -e^{-s} \alpha_- + e^s \alpha_+$$ are positively oriented Liouville forms on $\R_s \times M$.
\end{thmintro}

In the terminology of~\cite[Definition 1]{MNW}, we say that a pair of contact forms $(\alpha_-, \alpha_+)$ on a manifold $M$ is a \textbf{Liouville pair} if the $1$-form $$ \lambda \coloneqq e^{-s} \alpha_- + e^s \alpha_+$$ is a positively oriented Liouville form on $\R_s \times M$. By positively oriented, we mean that the volume form $d\lambda \wedge d\lambda$ is compatible with the natural orientation on $\R \times M$ induced by the natural orientation on $\R$ and the orientation on $M$. 

\begin{warn}
At first glance, Theorem~\ref{thmintro:supporting} seems almost identical to Hozoori's theorem. However, we warn the reader that the condition on $(\alpha_-, \alpha_+)$ that we consider is \textbf{different} than Hozoori's one. Indeed, there exist pairs of contact forms $(\alpha_-, \alpha_+)$ which are Liouville pairs as defined above, but such that $$(1-t) \alpha_- + (1+t) \alpha_+$$ is \textbf{not} a Liouville form on $[-1,1]_t \times M$; see Lemma~\ref{lem:alvslal}. It turns out that our condition  enjoys some nice symmetries (see Lemma~\ref{lem:balanced}) which make it much easier to work with. For instance, our notion of Liouville pair is easier to characterize than Hozoori's one (compare Lemma~\ref{lem:charanosov} which involves a single equation between three quantities, and Lemma~\ref{lem:charanosovlin} which involves two independent equations between four quantities). More importantly, \textbf{we do not know if our main results (Theorem~\ref{thmintro:homeq} and Theorem~\ref{thmintro:louvstr} below) are true for Hozoori's notion of Liouville pair.} The corresponding computations are much more complicated because of their lack of symmetry. 
\end{warn}

Theorem~\ref{thmintro:supporting} motivates the following

\begin{defn} \label{def:anoliouv}
An \textbf{Anosov Liouville pair} (AL pair for short) on an oriented $3$-manifold $M$ is a pair of contact forms $(\alpha_-, \alpha_+)$ such that both $(\alpha_-, \alpha_+)$ and $(-\alpha_-, \alpha_+)$ are Liouville pairs. We denote by $\mathcal{AL} \coloneqq \mathcal{AL}(M) \subset \Omega^1(M) \times \Omega^1(M)$ the space of Anosov Liouville pairs on $M$.
\end{defn}

Notice that we do not assume that $\xi_\pm \coloneqq \ker \alpha_\pm$ are transverse, since this is implied by the Liouville conditions; see Proposition~\ref{prop:anoliouv}. By Theorem~\ref{thmintro:supporting}, the intersection $\xi_- \cap \xi_+$ is spanned by an Anosov vector field. A positive time reparametrization of an Anosov flow remains Anosov, and we denote by $\mathcal{AF} \coloneqq \mathcal{AF}(M)$ the space of smooth oriented Anosov flows on $M$ up to positive time reparametrization. Alternatively, $\mathcal{AF}$ can be viewed as the space of smooth unit Anosov vector fields on $M$ for an arbitrary Riemannian metric on $M$, or the space of smooth $1$-dimensional oriented foliations spanned by Anosov vector fields on $M$, together with some extra orientation data. Hence, there is a natural continuous \emph{intersection map}
$$
\begin{array}{rccc}
\mathcal{I} : & \mathcal{AL} & \longrightarrow & \mathcal{AF} \\
		   & (\alpha_-, \alpha_+) & \longmapsto & \ker \alpha_- \cap \ker \alpha_+
\end{array} $$
which sends an AL pair to the $1$-dimensional (oriented) distribution obtained by intersecting the underlying contact structures. Here, we endow the spaces $\mathcal{AL}$ and $\mathcal{AF}$ with the $\mathcal{C}^\infty$ topology. Denoting by $\mathcal{BC}$ the space of smooth bi-contact structures on $M$ and by $\mathbb{P}\mathcal{AF}$ the space of smooth oriented projectively Anosov flows on $M$ up to positive time reparametrization, we have a similar intersection map
$$
\begin{array}{rccc}
\mathbb{P}\mathcal{I} : & \mathcal{BC} & \longrightarrow & \mathbb{P}\mathcal{AF} \\
		   & (\xi_-, \xi_+) & \longmapsto & \xi_- \cap \xi_+
\end{array} $$
as well as a \emph{kernel map}
$$
\begin{array}{rccc}
\underline{\mathrm{ker}} : & \mathcal{AL} & \longrightarrow & \mathcal{BC} \\
		   & (\alpha_-, \alpha_+) & \longmapsto & (\ker \alpha_-, \ker \alpha_+)
\end{array} $$

The main results of this paper, answering the Questions (1), (2) and (3) above, can be summarized as follows.

\begin{thmintro} \label{thmintro:homeq}
The maps in the commutative diagram
$$ \begin{tikzcd}[row sep=large]
\mathcal{AL} \arrow[r, "\underline{\ker}"] \arrow[d, "\mathcal{I}"']		&		\mathcal{BC} \arrow[d, "\mathbb{P}\mathcal{I}"] \\
\mathcal{AF} \arrow[hookrightarrow,r]	&		\mathbb{P}\mathcal{AF}
\end{tikzcd}$$
satisfy the following properties.
\begin{itemize}
\item $\mathcal{I}$ and $\mathbb{P}\mathcal{I}$ are acyclic Serre fibrations (Theorem~\ref{thm:homeq} and Theorem~\ref{thm:homeqproj}).
\item $\underline{\ker}$ is an acyclic Serre fibration onto its image (Theorem~\ref{thm:hominc}).
\item The inclusion $\underline{\mathrm{ker}}(\mathcal{AL}) \subset \mathbb{P}\mathcal{I}^{-1}\big(\mathcal{AF}\big)$ is strict in general (Theorem~\ref{thm:supporting}), but it is a homotopy equivalence (Theorem~\ref{thm:hominc'}).
\end{itemize}
\end{thmintro}

Recall that an acyclic Serre fibration is a Serre fibration which is also a weak homotopy equivalence, or equivalently, whose fibers are weakly contractible. All the topological spaces under consideration have the homotopy type of a CW complex (see the beginning of Section~\ref{sec:spaces}), so these acyclic Serre fibrations are homotopy equivalences by the Whitehead theorem. Unpacking the notations, \begin{itemize}
\item $\underline{\mathrm{ker}}(\mathcal{AL})$ is the space of bi-contact structures $(\xi_-, \xi_+)$ admitting contact forms $\alpha_-, \alpha_+$ such that $(\alpha_-, \alpha_+)$ is an AL pair,
\item $\mathbb{P}\mathcal{I}^{-1}\big(\mathcal{AF}\big)$ is the space of bi-contact structures supporting an Anosov flow.
\end{itemize}

We emphasize that the top row in the diagram of Theorem~\ref{thmintro:homeq} \emph{only involves concepts from contact and symplectic geometry.} This enables us to identify projectively Anosov flows with bi-contact structures, and Anosov flows with bi-contact structures satisfying a quantitative constraint, coming from the existence of a suitable pair of contact forms. Moreover, the space of AL pairs for a fixed underlying bi-contact structure is (weakly) contractible if non-empty. Hence, AL pairs can be thought of as \emph{auxiliary data attached to bi-contact structures}.

Our results can be summarized by the following slogan:

\bigskip

    \noindent\fbox{%
    \parbox{\linewidth - 3\fboxsep}{%
    The topological properties of the spaces $\mathcal{AF}$, $\mathbb{P}\mathcal{AF}$ and the inclusion $\mathcal{AF} \subset \mathbb{P}\mathcal{AF}$ can be translated into topological properties of the spaces $\mathcal{AL}$, $\mathcal{BC}$, and the map $\underline{\mathrm{ker}} : \mathcal{AL} \rightarrow \mathcal{BC}$, and vice versa.
    }%
    }

\bigskip

One important missing piece in this correspondence between Anosov dynamics and contact topology is the mirror notion of \emph{topological} or \emph{orbit equivalence} of flows in the contact world.

\begin{defn} Two Anosov flows $\Phi = \{ \phi^t\}$ and $\Psi= \{ \psi^t\}$ on $M$ are \textbf{topologically equivalent}, or \textbf{orbit equivalent}, if there exist a homeomorphism $h: M \rightarrow M$ and a continuous map $\tau : \R \times M \rightarrow \R$ such that $\tau(t,x) \geq 0$ for $t \geq 0$, and $$\psi^{\tau(t,x)} = h \circ \phi^t \circ h^{-1}(x)$$ for every $t \in \R$ and $x \in M$.
\end{defn}

In other words, the topological equivalence $h$ sends the oriented trajectories of $\phi$ onto the oriented trajectories of $\psi$, but does not necessarily preserves the parametrization. The structural stability of Anosov flows with smooth dependence on parameters~\cite[Theorem A.1]{LMM} implies that two smooth Anosov flows which are homotopic through smooth Anosov flows are topologically equivalent through a topological equivalence which is isotopic to the identity. We do not know if the converse is true.

\begin{question} If two (smooth) Anosov flows on $M$ are topologically equivalent (via a topological equivalence which is merely continuous), what can be said about the spaces of AL pairs supporting them? How to characterize topological equivalence in terms of AL pairs?
\end{question}

It is not clear to us how the (hyper)tight contact structure $\xi_\pm$ associated with a Anosov flow behave under topological equivalence. Solving these questions could have a significant impact in the understanding of Anosov flows from the perspective of contact geometry. For instance, a fundamental problem in $3$-dimensional Anosov dynamics is the following

\begin{question}[\cite{B17}] \label{quest:finite} On a closed $3$-manifold, are there finitely many Anosov flows up to topological equivalence?
\end{question}

It is known by the work of Colin, Giroux and Honda~\cite{CGH} that an atoroidal $3$-manifold carries finitely many isotopy classes of \emph{tight} contact structures. Although toroidal (and irreducible) $3$-manifolds can carry infinitely many isotopy classes of tight contact structures, all of them can be obtained from finitely many contact structures by performing \emph{Lutz twists} along suitable tori; see~\cite{CGH}. The authors also show that there are finitely many tight contact structures for a prescribed \emph{Giroux torsion}, up to isotopy and Dehn twists. Since the contact structures defined by (Anosov) Liouville pairs are by definition exactly semi-fillable, they are strongly fillable~\cite[Corollary 1.4]{E04}, hence they have zero Giroux torsion~\cite[Corollary 3]{G06}. This observation plays an essential role in the recent solution of Question~\ref{quest:finite} for the class of \emph{$\R$-covered} Anosov flows~\cite{BM21, M23}.

We hope that this \emph{coarse classification} of tight contact structures together with our homotopy correspondence could lead to important results in the classification of Anosov flows on $3$-manifolds. To this end, it is crucial to understand the following

\begin{question} Let $(\alpha_-, \alpha_+)$ be an AL pair on $M$. Fixing $\alpha_+$, what can be said about the Anosov flow supported by an AL pair $(\alpha'_-, \alpha_+)$, where $\alpha'_-$ is isotopic to $\alpha_-$?
\end{question}

The main difficulty here is that a path $(\alpha^t_-)_{t\in [0,1]}$ of contact forms from $\alpha^0_- = \alpha_-$ to $\alpha^1_- = \alpha'_-$ might not induce a path of bi-contact structures, as $\xi^t_- = \ker \alpha^t_-$ and $\xi_+ = \ker \alpha_+$ might fail to be transverse for some $t \in (0,1)$. Even if transversality holds, $(\alpha^t_-, \alpha_+)$ might fail to be an AL pair. Nevertheless, one could try to analyze the failure of these properties for a generic path $(\alpha^t_-)_t$, and apply suitable modifications to it. We wish to explore this direction in future work.

A closely related question, already raised by Hozoori~\cite[Question 7.2]{H22a} is the following.

\begin{question} Let $\Phi_0$ and $\Phi_1$ be two Anosov flows on $M$ and assume that they are homotopic through \emph{projectively} Anosov flows. Equivalently, assume that there exist two AL pairs $(\alpha_-^0, \alpha_+^0)$ and $(\alpha_-^1, \alpha_+^1)$ supporting $\Phi_0$ and $\Phi_1$, respectively, such that their underlying bi-contact structures are homotopic (through bi-contact structures). Are $\Phi_0$ and $\Phi_1$ homotopic through Anosov flows, i.e., are $(\alpha_-^0, \alpha_+^0)$ and $(\alpha_-^1, \alpha_+^1)$ homotopic through AL pairs? Are $\Phi_0$ and $\Phi_1$ topologically equivalent? 
\end{question}

\medskip

From the point of view of Liouville geometry, it is natural to weaken the definition of AL pairs as follows.

\begin{defn} A Liouville pair $(\alpha_-, \alpha_+)$ on $M$ is a \textbf{weak Anosov Liouville pair} (wAL pair for short) if it satisfies the following two conditions.
\begin{enumerate}[label=(\arabic*)]
	\item The contact plane fields $\xi_\pm \coloneqq \ker \alpha_\pm$ are everywhere transverse,
	\item The intersection $\xi_- \cap \xi_+$ is spanned by an Anosov vector field.
\end{enumerate}

An \textbf{Anosov Liouville structure} (AL structure for short) on $V=\R_s \times M$ is a pair $(\omega, \lambda)$ where $\omega = d \lambda$ is a symplectic form and $$\lambda =  e^{-s} \alpha_- + e^s \alpha_+$$ for a weak Anosov Liouville pair $(\alpha_-, \alpha_+)$.
We call the triple $(V, \omega, \lambda)$ an \textbf{Anosov Liouville manifold}.

An Anosov flow $\Phi$ is \textbf{supported} by the AL structure $(\omega, \lambda)$ if the vector field $X$ generating $\Phi$ satisfies $X \in \xi_- \cap \xi_+$.
\end{defn}

Note that the definition of wAL pairs \emph{does} make reference to the underlying Anosov flow, as opposed to AL pairs. By Theorem~\ref{thmintro:supporting}, there is an inclusion $\mathcal{AL} \subset \mathcal{AL}^w$, where $\mathcal{AL}^w$ denotes the space of wAL pairs on $M$. This inclusion is strict in general. The map $\mathcal{I}$ naturally extends to a map $\mathcal{I}^w : \mathcal{AL}^w \rightarrow \mathcal{AF}$, and similarly to the first bullet of Theorem~\ref{thmintro:homeq}, we have:

\begin{thmintro} \label{thmintro:louvstr}
The map $\mathcal{I}^w : \mathcal{AL}^w \rightarrow \mathcal{AF}$ is an acyclic Serre fibration, hence a homotopy equivalence.
\end{thmintro}

\begin{corintro} \label{cor} Let $\Phi_0$ and $\Phi_1$ be two Anosov flows on $M$, supported by AL structures $(\omega_0, \lambda_0)$ and $(\omega_1, \lambda_1)$, respectively. If $\Phi_0$ and $\Phi_1$ are homotopic through Anosov flows, then $(\omega_0, \lambda_0)$ and $(\omega_1, \lambda_1)$ are homotopic through AL structures, and $(V, \omega_0, \lambda_0)$ and $(V, \omega_1, \lambda_1)$ are exact symplectomorphic.
\end{corintro}

Here, an exact symplectomorphism $\psi : (V,\omega_0, \lambda_0) \rightarrow (V, \omega_1, \lambda_1)$ is a diffeomorphism such that $\psi^*\lambda_1 = \lambda_0 + df$ for some smooth function $f : V \rightarrow \R$. In Corollary~\ref{cor}, we can further assume that $df$ has compact support.

\begin{proof}[Proof of Corollary~\ref{cor}] If $\Phi_0$ and $\Phi_1$ are homotopic through Anosov flows, Theorem~\ref{thmintro:louvstr} provides a \emph{continuous} path of smooth AL structures from $(\omega_0, \lambda_0)$ to $(\omega_1, \lambda_1)$. This path can be smoothed while ensuring the existence of some number $A > 0$ such that the corresponding Liouville vector fields are all transverse to $\{\pm A\} \times M$. Then, \cite[Proposition 11.8]{CE12} provides an exact symplectomorphism $\psi$ such that $\psi^*\lambda_1 - \lambda_0$ is compactly supported.
\end{proof}

Anosov Liouville manifolds have numerous interesting invariants coming from Floer theory, e.g., symplectic cohomology and wrapped Fukaya category. As an important consequence of Corollary~\ref{cor}, these are \emph{invariants of the underlying Anosov flow}, and only depend on its homotopy class in the space of Anosov flows. Some of these invariants are studied in detail in~\cite{CLMM}. To our knowledge, this is the first thorough analysis of symplectic invariants of non-Weinstein Liouville manifolds.

\medskip

One can also consider Liouville pairs $(\alpha_-, \alpha_+)$ whose underlying contact planes are everywhere transverse. We call such pairs \textbf{transverse Liouville pairs}. They correspond to particular projectively Anosov flow that we call \textbf{semi-Anosov flows}, see Remark~\ref{rem:transversepair} below. General Liouville pairs (without the transversality assumption) are more complicated to understand, but their underlying contact planes can only intersect \emph{positively}, see Remark~\ref{rem:positiv} below. In the terminology of~\cite{CS11}, they constitute \textbf{positive contact pairs}. 

These geometric structures are summarized in the following diagram; the ones in blue are the main protagonists of this article. Liouville pairs and positive contact pairs will be investigated in forthcoming work~\cite{Mas}.
$$\centering
\begin{tikzcd}[column sep=0.5em]
\textcolor{blue}{\big\{\text{AL pairs}\big\}} \arrow[r, symbol=\subset, blue] \arrow[drr, blue, bend right=11, "\underline{\ker}"']& \textcolor{blue}{\big\{\text{Weak AL pairs}\big\}} \arrow[r, symbol=\subset] & \big\{\text{Transverse Liouville pairs}\big\} \arrow[r, symbol=\subset] \arrow[d, "\underline{\ker}"'] &[-10pt] \big\{\text{Liouville pairs}\big\} \arrow[d, "\underline{\ker}"']\\
				&						&	\textcolor{blue}{\big\{\text{Bi-contact structures}\big\}} \arrow[r, symbol=\subset] & \big\{\text{Positive contact pairs}\big\}
\end{tikzcd}$$

\paragraph*{Acknowledgments}
I am grateful to my PhD advisor John Pardon for his constant support and encouragement. I would like to thank Sergio Fenley, Jonathan Zung and Malo J\'{e}z\'{e}quel for insightful conversations about Anosov flows, and Surena Hozoori for multiple discussions about his work and fruitful exchanges. I am grateful to the anonymous referees for pointing out several typos and inaccuracies, as well as providing valuable suggestions.

        \section{Anosov Liouville pairs}

		  \subsection{Preliminary definitions}

If $X$ is a non-singular vector field on $M$, we write $$N_X \coloneqq TM \slash \langle X \rangle.$$
An orientation on $M$ naturally determines an orientation on the plane bundle $N_X \rightarrow M$. We denote by $\pi :TM \rightarrow N_X$ the quotient map. There is a correspondence between $n$-forms $\alpha$ on $M$ satisfying $\iota_X \alpha = 0$ and $n$-forms $\overline{\alpha}$ on $N_X$. Moreover, a vector field $Y$ on $M$ induces a section $\overline{Y}\coloneqq \pi(Y)$ on $N_X$. The operator $\mathcal{L}_X$, the Lie derivative along $X$, naturally induces an operator, still denoted by $\mathcal{L}_X$, on sections of $N_X$ and on $n$-forms on $N_X$.

\begin{defn} \label{def:bicont}
A \textbf{bi-contact structure} on an oriented $3$-manifold $M$ is a pair of co-oriented contact structures $(\xi_-, \xi_+)$ such that $\xi_-$ is negative, $\xi_+$ is positive and $\xi_-$ and $ \xi_+$ are transverse everywhere.

A non-singular flow $\Phi$ on $M$ generated by a vector field $X$ is \textbf{supported} by a bi-contact structure $(\xi_-, \xi_+)$ if $X \in \xi_- \cap \xi_+$, and the following orientation compatibility condition holds. Let $\overline{\xi}_\pm \subseteq N_X$ be the image of $\xi_\pm$ under the quotient map $\pi :TM \rightarrow N_X$. The orientations on $M$, $\xi_\pm$ and $X$ induce natural orientations on $N_X$ and $\overline{\xi}_\pm$. We require that the orientation on $N_X$ coincides with the one on $\overline{\xi}_- \oplus \overline{\xi}_+$ (see Figure~\ref{fig:bicontact} and Figure~\ref{fig:Anosovorientation}).

Similarly, $\Phi$ is supported by a (weak) Anosov Liouville pair $(\alpha_-, \alpha_+)$ if it is supported by the bi-contact structure $(\xi_-, \xi_+) = (\ker \alpha_-, \ker \alpha_+)$.
\end{defn}

Note that the definitions of bi-contact structures and (weak) Anosov Liouville pairs still make sense if $\xi_\pm$, or $\alpha_\pm$, are merely $\mathcal{C}^1$. We will always assume that bi-contact structures and (weak) Anosov Liouville pairs are smooth unless stated otherwise. Bi-contact structures and (weak) Anosov Liouville pairs obviously constitute open subsets of the space of pairs of $2$-plane fields on $M$ and the space of pairs of $1$-forms on $M$, respectively, since they are defined by open conditions.

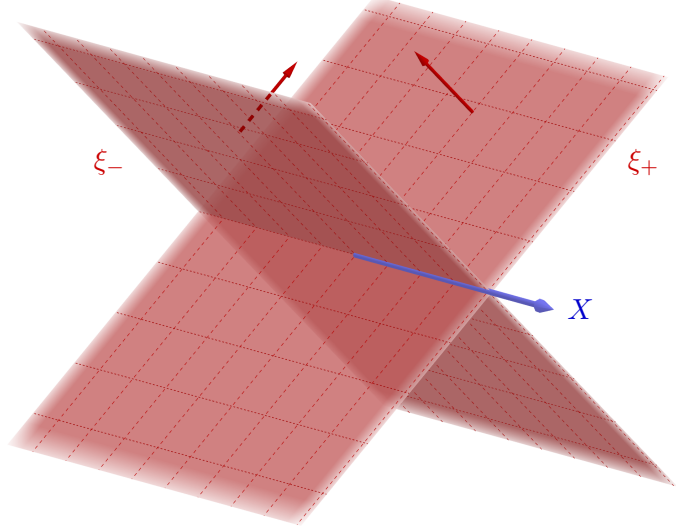
\begin{figure}[t]
    \begin{center}
        \begin{picture}(90, 72)(0,0)
        \put(0,0){\includegraphics[width=90mm]{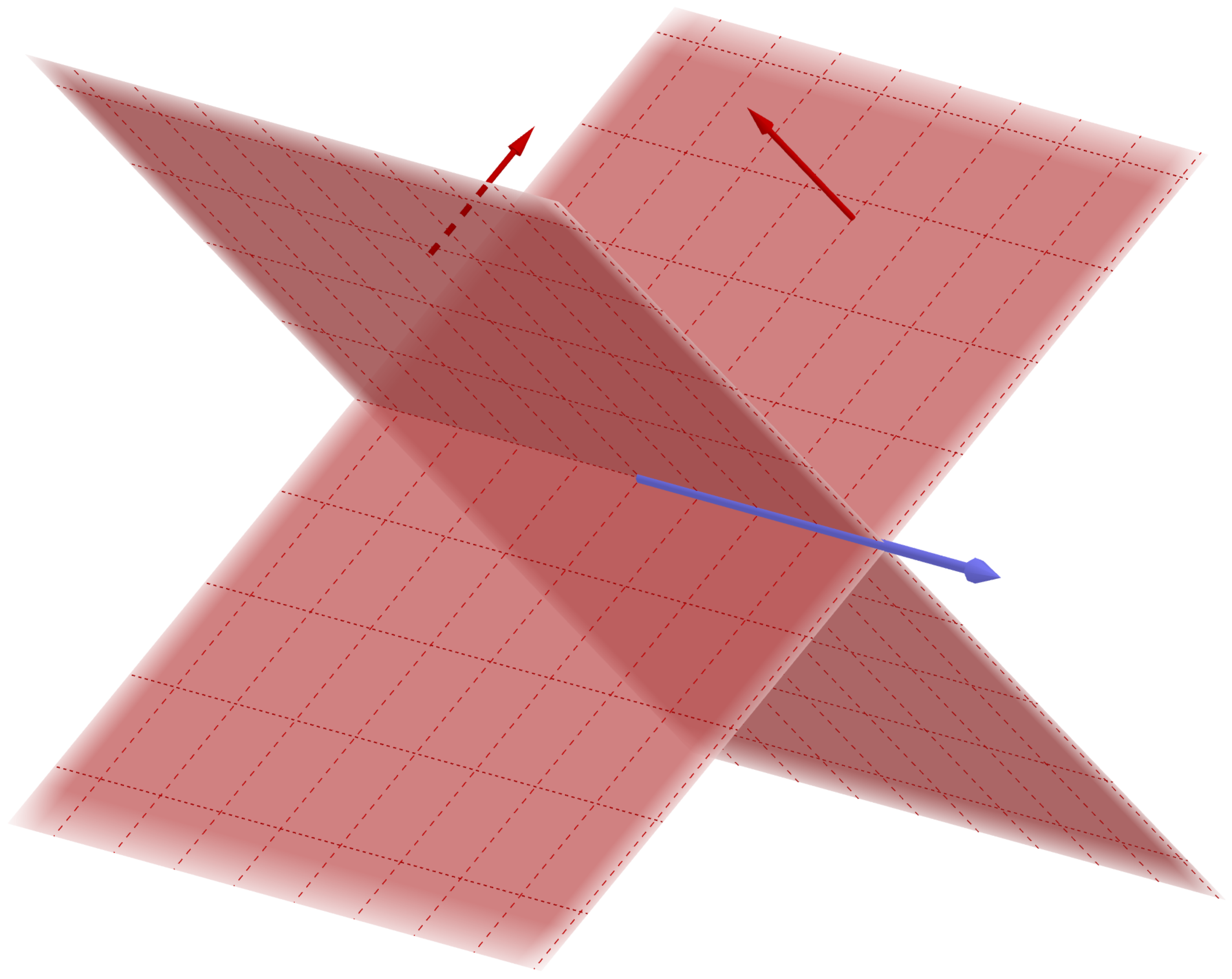}}
        \color{red1}
        \put(12,48){$\xi_-$}
        \put(83,48){$\xi_+$}
        \color{blue1}
        \put(75,28){$X$}
        \end{picture}
        \captionsetup{width=100mm}
        \caption{Coorientation convention for bi-contact structures supporting a vector field or a flow.}
    \label{fig:bicontact}
    \end{center}
\end{figure}

If $(\xi_-, \xi_+)$ is a bi-contact structure supporting a non-singular flow $\Phi = \{\phi^t\}$, then the bi-contact structure obtained from $(\xi_-, \xi_+)$ by reversing the coorientations of both $\xi_-$ and $\xi_+$ supports $\Phi$ as well. Reversing the coorientation of $\xi_-$ or $\xi_+$ only yields a bi-contact structure supporting the reversed flow $\Phi^{-1} = \{\phi^{-t}\}$.

It is easy to deduce from Theorem~\ref{thmintro:supporting} the very well-known

\begin{cor} The space of (smooth, $\mathcal{C}^1$) Anosov vector fields on $M$ is open in the $\mathcal{C}^1$ topology.
\end{cor}

\begin{proof}
Let $X$ be an Anosov vector field on $M$ and $(\alpha_-, \alpha_+)$ be an AL pair supporting $X$. We choose a $1$-form $\theta$ such that $\theta(X) \equiv 1 $. If $X'$ is another vector field which is sufficiently $\mathcal{C}^1$-close to $X$, the pair $(\alpha'_-, \alpha'_+)$ defined by
\begin{align*}
\alpha'_\pm \coloneqq \alpha_\pm - \frac{\alpha_\pm(X')}{\theta(X')} \theta
\end{align*}
is an AL pair supporting $X'$ and by Theorem~\ref{thmintro:supporting}, $X'$ is Anosov.
\end{proof}

If $(\alpha_-, \alpha_+)$ is an AL pair on $M$ and $\sigma: M \rightarrow \R$ is a smooth function, it follows from the definition that $$\sigma \cdot (\alpha_-, \alpha_+) \coloneqq \big(e^{-\sigma} \alpha_-, e^{\sigma} \alpha_+\big)$$
is also an AL pair that defines the same bi-contact structure as $(\alpha_-, \alpha_+)$. These two AL pairs will be called \textbf{equivalent}. This defines  an action of $\mathcal{C}^\infty(M, \R)$ on the space of AL pairs.

\begin{defn}
A pair of contact forms $(\alpha_-, \alpha_+)$ on M, negative and positive, respectively, is \textbf{balanced} if $$ \alpha_+ \wedge d\alpha_+ = - \alpha_- \wedge d\alpha_-.$$
\end{defn}

In other words, $(\alpha_-, \alpha_+)$ is balanced if $\alpha_\pm$ define opposite volume forms on $M$.

\begin{lem} \label{lem:balanced}
Two equivalent AL pairs on $M$ define Liouville isomorphic Liouville structures on $\R \times M$. Any AL pair on $M$ is equivalent to a (unique) balanced one.
\end{lem}

\begin{proof} Let $(\alpha_-, \alpha_+)$ be an AL pair on $M$ and $\lambda \coloneqq  e^{-s} \alpha_- + e^s \alpha_+$ be the corresponding Liouville form. If $\sigma \in \mathcal{C}^\infty(M, \R)$ and $\lambda' \coloneqq e^{-(s+\sigma)} \alpha_- + e^{s+ \sigma} \alpha_+$, the diffeomorphism 
$$\begin{array}{rccc}
\Psi :  & \R \times M & \longrightarrow & \R \times M \\
	 & (s,x) & \longmapsto & (s - \sigma(x), x)
\end{array}$$
satisfies $\Psi^* \lambda' = \lambda$. Moreover, if $f: M \rightarrow \R_{>0}$ is such that $$ \alpha_- \wedge d\alpha_- = -f\, \alpha_+ \wedge d\alpha_+,$$ then $\sigma \cdot ( \alpha_-,  \alpha_+)$ is balanced if and only if $\sigma = \frac{1}{4} \ln f$.
\end{proof}

As a straightforward application of Gray's stability theorem and the above lemma, we have the following

\begin{lem} Let $(\alpha_-, \alpha_+)$ be an AL pair and let $\xi_+ \coloneqq \ker \alpha_+$. If $\xi'_+ = \ker \alpha'_+$ is a contact structure homotopic to $\xi_+$, then there exists a path of AL pairs $(\alpha^t_-, \alpha^t_+)$, $t \in [0,1]$, such that $(\alpha^0_-, \alpha^0_+) = (\alpha_-, \alpha_+)$ and $\alpha^1_+ = \alpha'_+$.
\end{lem}

\begin{defn} \label{def:closed} A pair of contact forms $(\alpha_-, \alpha_+)$ on M, negative and positive, respectively, is \textbf{closed} if $\alpha_- \wedge \alpha_+$ is a closed $2$-form.
\end{defn}

It is straightforward to check that a closed pair $(\alpha_-, \alpha_+)$ is an AL pair (see also Lemma~\ref{lem:charanosov} below). As we will see in Proposition~\ref{prop:volume}, closed AL pairs are in correspondence with volume preserving Anosov flows.

		\subsection{Elementary properties of Anosov Liouville pairs} \label{section:elem}
		
The notion of Anosov Liouville pair can be conveniently characterized in the following way, which only involves the forms and their exterior differentials.

\begin{lem} \label{lem:charanosov}
Let $(\alpha_-, \alpha_+)$ be a pair of $1$-forms on $M$. We  write
\begin{align*}
\alpha_+\wedge d \alpha_+ &= f_+ \, \mathrm{dvol}, \\
\alpha_- \wedge d \alpha_- &= - f_- \, \mathrm{dvol}, \\
d(\alpha_- \wedge \alpha_+) &= f_0 \ \mathrm{dvol},
\end{align*}
where $\mathrm{dvol}$ is any volume form on $M$ and $f_\pm, f_0 : M \rightarrow \R$ are smooth functions. Then $(\alpha_-, \alpha_+)$ is an AL pair if and only if $f_\pm > 0$, and
\begin{align} \label{ineqfexp}
f_0^2 < 4 f_- f_+.
\end{align}
\end{lem}

\begin{proof}
Following~\cite[Lemma 9.4]{MNW}, $(\alpha_-, \alpha_+)$ is a Liouville pair if and only if for all constants $C_-, C_+ \geq 0$ with $(C_-,C_+) \neq (0,0)$, $$(C_+ \alpha_+ - C_- \alpha_-)\wedge (C_+ d\alpha_+ + C_- d\alpha_-) > 0,$$
which is equivalent to $$C_+^2 f_+ + C_- C_+ f_0 + C_-^2 f_- > 0.$$ Applying this fact to $(\alpha_-, \alpha_+)$ and $(- \alpha_-, \alpha_+)$, we obtain that $(\alpha_-, \alpha_+)$ is an AL pair if and only if $f_\pm > 0$ and for every $x \in \R$, $$x^2 f_+ + x f_0 + f_- > 0,$$ which is equivalent to~\eqref{ineqfexp} by the quadratic formula.
\end{proof}

\begin{rem} \label{rem:liouv}
The proof also shows that $(\alpha_-, \alpha_+)$ is a Liouville pair if and only if $f_\pm > 0$ and $- f_0 < 2 \sqrt{f_- f_+}$.
\end{rem}

We now use this criterion to show some natural  geometric properties of Anosov Liouville pairs.

\begin{prop} \label{prop:anoliouv}
Let $(\alpha_-, \alpha_+)$ be an Anosov Liouville pair. Then it defines a bi-contact structure $(\xi_-,\xi_+) = (\ker \alpha_-, \ker \alpha_+)$. Moreover, if $X \in \xi_- \cap \xi_+$ is a nowhere vanishing vector field and $R_\pm$ is the Reeb vector field of $\alpha_\pm$, then $\{X, R_-, R_+\}$ is a basis at every point of $M$.
\end{prop}

\begin{proof}
We first show that $\xi_-$ and $\xi_+$ intersect transversally everywhere. Assume by contradiction that there exist a point $x \in M$ and two linearly independent vectors $X, Y \in T_xM$ such that $\alpha_\pm(X) = \alpha_\pm(Y) = 0$. In what follows, all the quantities will be implicitly evaluated at this point $x$. We can assume without loss of generality that $d \alpha_+(X,Y) >0$ and $\mathrm{dvol}(X,Y,R_+)=1$. We compute
\begin{align*}
\alpha_+ \wedge d\alpha_+ (X,Y,R_+) &= d\alpha_+(X,Y) = f_+, \\
\alpha_- \wedge d\alpha_- (X,Y,R_+) &= \alpha_-(R_+) d\alpha_-(X,Y) = -f_-,\\
\alpha_- \wedge d\alpha_+ (X,Y,R_+) &= \alpha_-(R_+) d\alpha_+(X,Y),\\
\alpha_+ \wedge d\alpha_- (X,Y,R_+) &= d\alpha_-(X,Y),
\end{align*}
hence
\begin{align*}
f_0^2 - 4f_-f_+ &= 
\begin{multlined}[t] (d\alpha_-(X,Y))^2 - 2 \alpha_-(R_+) d\alpha_-(X,Y) d\alpha_+(X,Y) \\+ \alpha_-(R_+)^2 (d\alpha_+(X,Y))^2  
+4 \alpha_-(R_+) d\alpha_-(X,Y) d\alpha_+(X,Y)
\end{multlined}\\
&= \left(d\alpha_-(X,Y) + \alpha_-(R_+) d\alpha_+(X,Y) \right)^2\\
&\geq 0,
\end{align*}
contradicting~\eqref{ineqfexp}.

For the second part, we write
\begin{align*}
\alpha_- \wedge d\alpha_+ = g_+ \, \mathrm{dvol}, \\
\alpha_+ \wedge d\alpha_- = g_- \, \mathrm{dvol},
\end{align*}
where $g_\pm : M \rightarrow \R$ are smooth functions (note that $f_0 = g_- - g_+$) and we compute
\begin{align*}
\alpha_+ \wedge d\alpha_+ (X,R_+,\cdot\, ) &= -d\alpha_+(X,\cdot\, ) = f_+\, \mathrm{dvol}(X,R_+, \cdot\, ), \\
\alpha_- \wedge d\alpha_- (X,R_+, \cdot\, ) &= -\alpha_-(R_+) d\alpha_-(X,\cdot\, ) + d\alpha_-(X, R_+)\alpha_-  = -f_-\, \mathrm{dvol}(X,R_+, \cdot\, ),\\
\alpha_- \wedge d\alpha_+ (X,R_+, \cdot\, ) &= - \alpha_-(R_+) d\alpha_+(X, \cdot\, ) = g_+\,  \mathrm{dvol}(X, R_+, \cdot\, ),\\
\alpha_+ \wedge d\alpha_- (X, R_+, \cdot\, ) &=  - d\alpha_-(X,\cdot\, ) + d\alpha_-(X, R_+) \alpha_+  = g_-\, \mathrm{dvol}(X, R_+, \cdot\, ).
\end{align*}
Let us assume that $\mathrm{dvol}(X, R_-, R_+)=0$ at a point $x \in M$. In what follows, all the quantities will be implicitly evaluated at this point $x$. Plugging in $R_-$ in the first two of the four equations above yields $$d\alpha_-(X,R_+) = d\alpha_+(X, R_-) = 0.$$
Note that $X$ and $R_+$ are not colinear since $\alpha_+(X)=0$ and $\alpha_+(R_+)=1$. The last two of the four equations above imply $\alpha_-(R_+) \neq 0$ and 
\begin{align*}
f_+ &= \frac{1}{\alpha_-(R_+)} g_+, \\
f_- &= -\alpha_-(R_+) g_-.
\end{align*}
Finally, 
\begin{align*}
f_0^2 - 4 f_- f_+ &= \big(g_- - g_+ \big)^2 + 4 g_-g_+ \\
&= \big(g_- + g_+\big)^2 \\
&\geq 0,
\end{align*}
contradicting~\eqref{ineqfexp}.
\end{proof}

\begin{rem} \label{rem:positiv}
A (non-Anosov) Liouville pair may not define a bi-contact structure, namely $\xi_-=\ker \alpha_-$ and $\xi_+ = \ker \alpha_+$ may not be transverse everywhere. Nevertheless, the first part of the proof can easily be adapted to show that at a point where $\xi_-$ and $\xi_+$ coincide, their orientations coincide (and their coorientations are opposite). In the terminology of~\cite{CS11}, $(\xi_-, \xi_+)$ is a \emph{positive pair} of contact structures. After a generic perturbation of $\alpha_-$ and/or $\alpha_+$, the singular set $\Delta \coloneqq \{ x \in M : \xi_-(x) = \xi_+(x)\}$ is a smoothly embedded link in $M$. Moreover, it can be shown that $f_0 > 0$ along $\Delta$, so the Liouville condition of Remark~\ref{rem:liouv} is largely satisfied. We refer to our forthcoming article~\cite{Mas} for detailed proofs of these facts and a thorough investigation of general Liouville pairs.
\end{rem}

For any AL pair $(\alpha_-, \alpha_+)$, if $X$ (or $\mathrm{dvol}$) is chosen so that $\mathrm{dvol}(X,R_-,R_+) = 1$, then
\begin{align*}
f_+ &= d\alpha_+(X, R_-) = \mathcal{L}_X \alpha_+ (R_-), \\
f_- &= d\alpha_-(X, R_+) = \mathcal{L}_X \alpha_- (R_+),\\
g_+ &= \alpha_-(R_+) f_+, \\
g_- &= - \alpha_+(R_-) f_-.
\end{align*}
Moreover, if $(\alpha_-, \alpha_+)$ is balanced, i.e., if $f_+ = f_-$, the condition~\eqref{ineqfexp} becomes
\begin{align} \label{eqReebexp}
|\alpha_-(R_+) + \alpha_+(R_-)| < 2.
\end{align}
In fact, (balanced) AL pairs can be completely characterized by their Reeb vector fields.

\begin{prop} \label{prop:reebvf}
Let $(\alpha_-, \alpha_+)$ be a pair of contact forms on $M$, negative and positive, respectively, and assume that it is balanced. Then it is an AL pair if and only if~\eqref{eqReebexp} is satisfied.
\end{prop}

\begin{proof}
We only have to show that under these hypothesis, the conclusions of Proposition~\ref{prop:anoliouv} are satisfied, since these imply that $g_+ = \alpha_-(R_+) f_+$ and $g_- = - \alpha_+(R_-) f_-$ and Lemma~\ref{lem:charanosov} concludes the proof.

Assume first that $\xi_-$ and $\xi_+$ are not transverse at a point $x\in M$. With the same notations as in the proof of Proposition~\ref{prop:anoliouv}, similar computations show that at this point, 
\begin{align*}
\alpha_+ \wedge d\alpha_+ (X,Y,R_+) &= d\alpha_+(X,Y) = f_+, \\
\alpha_- \wedge d\alpha_- (X,Y,R_+) &= \alpha_-(R_+) d\alpha_-(X,Y) = -f_-,\\
\alpha_+ \wedge d\alpha_+(X,Y,R_-) &= \alpha_+(R_-) d\alpha_+(X,Y) = f_+ \, \mathrm{dvol}(X,Y,R_-), \\
\alpha_- \wedge d\alpha_-(X,Y,R_-) &= d\alpha_-(X,Y) = -f_- \, \mathrm{dvol}(X,Y,R_-),
\end{align*}
hence by the first and third equalities,
$$\mathrm{dvol}(X,Y,R_-) = \alpha_+(R_-),$$ 
and by the second and fourth equalities,
$$\alpha_-(R_+) \alpha_+(R_-) = 1,$$
contradicting~\eqref{eqReebexp} by the inequality of arithmetic and geometric means.

Assuming now that $\mathrm{dvol}(X,R_-, R_+)=0$ at a point $x \in M$, the proof of Proposition~\ref{prop:anoliouv} showed that at this point, $$d\alpha_-(X,R_+) = d\alpha_+(X, R_-) = 0,$$
and $$g_+ =\alpha_-(R_+) f_+.$$
Similarly,
\begin{align*}
\alpha_+ \wedge d\alpha_+ (X,R_-,\cdot\, ) &= -\alpha_+(R_-) d\alpha_+(X,\cdot\, ) = f_+\, \mathrm{dvol}(X,R_-, \cdot\, ), \\
\alpha_- \wedge d\alpha_+ (X,R_-, \cdot\, ) &= - d\alpha_+(X, \cdot\, ) = g_+\,  \mathrm{dvol}(X, R_-, \cdot\, ),
\end{align*}
hence $$f_+ = \alpha_+(R_-) g_+.$$
Once again, we obtain that $$\alpha_-(R_+) \alpha_+(R_-) = 1,$$ contradicting~\eqref{eqReebexp}.
\end{proof}

    \section{From Anosov flows to Anosov Liouville pairs and back}
    
    In this section, we adapt the proof of~\cite[Theorem 1.1]{H22a} to the setting of Anosov Liouville pairs as defined in the Introduction.
    
    	\subsection{Anosov and projectively Anosov flows} \label{sec:anosov}

We recall the definitions of Anosov and projectively Anosov flows with an emphasis on our orientation conventions, and recast them in terms of the existence of suitable $1$-forms.

\begin{defn} \label{definition:defanosov}
Let $\Phi = \{\phi^t\}_{t \in \R}$ be a flow on $M$ generated by a non-singular $\mathcal{C}^1$ vector field $X$.
\begin{itemize}
\item $\Phi$ is \textbf{Anosov} if there exists a continuous invariant \textbf{hyperbolic splitting}
\begin{align} \label{anosovsplit}
TM = \langle X \rangle \oplus E^s \oplus E^u
\end{align}
where $E^s, E^u$ are $1$-dimensional bundles such that for some (any) Riemannian metric $g$ on $M$, there exist constants $C, a >0$ such that for all $v \in E^s$ and $t \geq 0$, $$\Vert d\phi^t(v)\Vert \leq C e^{-at} \Vert v \Vert,$$
and for all $v \in E^u$ and $t \geq 0$, $$\Vert d\phi^t(v)\Vert \geq C e^{at} \Vert v \Vert.$$
$E^s$ and $E^u$ are called the (strong) stable and unstable bundles of $\Phi$, respectively.
\item $\Phi$ is \textbf{projectively Anosov} if there exists a continuous invariant splitting
\begin{align} \label{projanosovsplit}
TM \slash \langle X \rangle = N_X = \overline{E}^s \oplus \overline{E}^u
\end{align}
where $\overline{E}^s, \overline{E}^u$ are $1$-dimensional bundles such that for some (any) Riemannian metric $\overline{g}$ on $N_X$, there exist constants $C, a >0$ such that for all unit vectors $v_s \in \overline{E}^s, v_u \in \overline{E}^u$, and $t \geq 0$, $$ \Vert d\phi^t(v_u)\Vert \geq C e^{at} \, \Vert d\phi^t(v_s) \Vert.$$
Such a splitting is called a \textbf{dominated splitting}. We denote by $E^{ws} \coloneqq \pi^{-1}\big(\overline{E}^s\big)$ and $E^{wu} \coloneqq \pi^{-1}\big(\overline{E}^u\big)$ the weak-stable and weak-unstable bundles of $\Phi$, respectively.
\item In both cases, if the constant $C$ can be chosen to be $1$, the corresponding metrics $g$ and $\overline{g}$ are called \textbf{adapted} to $\Phi$. 
\item The Anosov (resp.~projectively Anosov) flow $\Phi$ is \textbf{oriented} if $E^s$ and $E^u$ are oriented (resp.~$\overline{E}^s$ and $\overline{E}^u$ are oriented) and their orientations are compatible with the splitting~\eqref{anosovsplit} (resp.~the splitting~\eqref{projanosovsplit}).
\end{itemize}
\end{defn}

Anosov flows are projectively Anosov, with dominated splitting $N_X = \pi\big(E^s\big) \oplus \pi\big(E^u\big)$. Every three dimensional (projectively) Anosov flow admits a smooth adapted metric, see~\cite[Proposition 5.1.5]{FH19}.\footnote{The proof is given for Anosov flows but it easily generalizes to projectively Anosov flows in dimension three; note that it is not sufficient to integrate an arbitrary metric along the flow for a large time!} Anosov famously showed that the weak and strong stable/unstable bundles of an Anosov flow are \emph{uniquely integrable}. Moreover, $E^{ws}$ and $E^{wu}$ integrate into \emph{taut foliations} $\mathcal{F}^{ws}$ and $\mathcal{F}^{wu}$, respectively. This is not true for projectively Anosov flows; see~\cite[Example 2.2.9]{ET}.

    \noindent\fbox{%
    \parbox{\linewidth - 3\fboxsep}{%
    \textbf{In the rest of the article, we implicitly assume that all of the Anosov and projectively Anosov flows under consideration are oriented.}
    }%
    }

\begin{figure}[t]
    \begin{center}
        \begin{picture}(60, 60)(0,0)
        \put(0,0){\includegraphics[width=60mm]{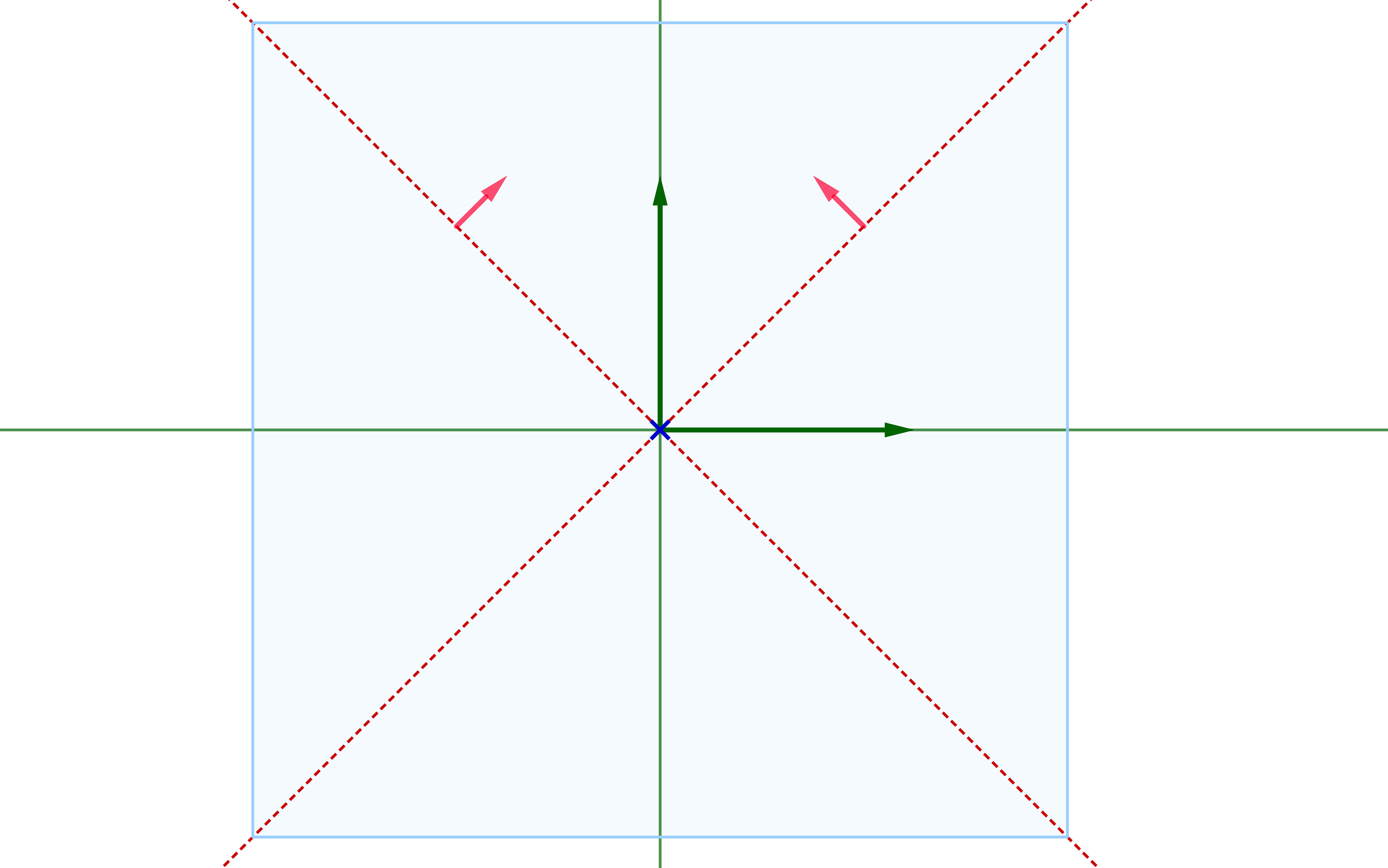}}
        \color{blue1}
        \put(26.5,24){$X$}
        \color{green1}
        \put(40,27){$e_s$}
        \put(52,26){$E^s$}
        \put(31,53){$E^u$}
        \put(31,41){$e_u$}
        \color{red1}
        \put(5,47){$\xi_-$}
        \put(52,47){$\xi_+$}
        \end{picture}
        \captionsetup{width=90mm}
        \caption{Orientation convention for the Anosov splitting. There is a similar picture for the dominated splitting. The vector field X points toward the reader, watch out.}
        \label{fig:Anosovorientation}
    \end{center}
\end{figure}
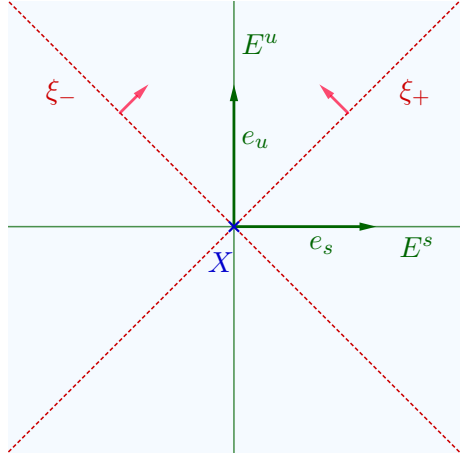

\bigskip

The following definition appears in~\cite[Definition 3.11]{H22a}, see also~\cite[Proposition 3.12]{H22a}.

\begin{defn} Let $\Phi$ be a projectively Anosov flow on $M$ generated by a vector field $X$ and $\overline{g}$ be a Riemannian metric on $N_X$. The \textbf{expansion rates} in the stable and unstable directions for $\overline{g}$ are continuous functions $r_s, r_u : M \rightarrow \R$ defined by $$ r_s \coloneqq \left. \frac{\partial}{\partial t}\right\rvert_{t=0} \ln \Vert d\phi^t (\overline{e}_s)\Vert, \qquad r_u \coloneqq \left. \frac{\partial}{\partial t}\right\rvert_{t=0} \ln \Vert d\phi^t (\overline{e}_u)\Vert,$$
where $\overline{e}_s$ and $\overline{e}_u$ are unit sections of $\overline{E}^s$ and $\overline{E}^u$, respectively, which are continuous and continuously differentiable along the flow $\Phi$.\footnote{In particular, the function $t \mapsto d\phi^t(\overline{e}_{s,u})$ is differentiable and has positive norm, so $r_{s,u}$ is well-defined.} Moreover, $$\mathcal{L}_X \overline{e}_s = -r_s \overline{e}_s, \qquad \mathcal{L}_X \overline{e}_u = -r_u \overline{e}_u.$$
\end{defn}

The Lie derivative above means the following: if $e_{s,u}$ is a vector field on $M$ which is a lift of $\overline{e}_{s,u}$ with the same regularity, the quantity $\mathcal{L}_X \overline{e}_{s,u} \coloneqq \pi\big(\mathcal{L}_X e_{s,u}\big)$ is a section of $\overline{E}^{s,u}$ which is independent of the choice of the lift. Here, $\mathcal{L}_X e_{s,u}$ denotes the usual Lie derivative along $X$, defined by
$$\mathcal{L}_X e_{s,u} \coloneqq \left. \frac{\partial}{\partial t}\right\rvert_{t=0} \big(\phi^t\big)^* e_{s,u}.$$ Notice that this involves the \emph{pullback} along $\phi^t$, and thus the differential of the \emph{inverse} of the flow, which explains the presence of negative signs in the previous formulae.

Many natural quantities (e.g., functions, vector fields, $1$-forms, Riemannian metrics) defined for (projectively) Anosov flows are continuous and can be upgraded to quantities which are continuous \emph{and} continuously differentiable along the flow by considering the averaging $$\frac{1}{T}\int_0^T \big(\phi^t\big)^*\_ \, dt$$ for some $T>0$. Nevertheless, these quantities may not be $\mathcal{C}^1$. It is therefore natural to consider the following spaces (only the cases $k=0,1$ and $n=0,1$ will be relevant for us).

\begin{defn} \label{def:reg}
Le $X$ be a smooth, non-singular vector field on $M$, and $k \geq 0$ be a non-negative integer.
\begin{itemize}
    \item A $n$-form $\alpha$ on $M$ is of class $\mathcal{C}^k_X$ if $\alpha$ is differentiable along $X$, and both $\alpha$ and $\mathcal{L}_X \alpha$ are of class $\mathcal{C}^k$. We denote by $\Omega^n_{X,k}$ the space of $n$-forms on $M$ of class $\mathcal{C}^k_X$ satisfying $\iota_X \alpha = 0$ (which is vacuous for $n=0$, i.e., for functions). We also denote by $\Omega^n_{X} = \Omega^n_{X, \infty} \subset \Omega^n$ the space of smooth $n$-forms satisfying $\iota_X \alpha = 0$.
    \item On $\Omega^n_{X,k}$, there is a natural norm defined by $$\vert \alpha \vert_{\mathcal{C}_X^k} \coloneqq \vert \alpha \vert_{\mathcal{C}^k} + \vert \mathcal{L}_X \alpha \vert_{\mathcal{C}^k},$$ making $\big(\Omega_{X,k}^n, \vert \cdot \vert_{\mathcal{C}_X^k}\big)$ a Banach space. $\Omega^n_{X}$ is naturally a Fr\'echet space as a closed subspace of $\Omega^n$
\end{itemize}
\end{defn}
These definitions naturally extend to sections of $N_X$ and $n$-forms on $N_X$.

In Appendix~\ref{appendixA}, we show some density results for these spaces which are particularly useful when dealing with (projectively) Anosov flows generated by $\mathcal{C}^1$ vector fields and can be used to bypass Hozoori's delicate approximation techniques in~\cite[Section 4]{H22a}. The results in Appendix~\ref{appendixA} are not needed (except in the proof of Theorem~\ref{thm:projcontract}) if we restrict our attention to smooth Anosov flows in view of Lemma~\ref{lem:c1} below.

\medskip

In dimension three, the definitions of Anosov and projectively Anosov flows can be rephrased in terms of the existence of certain $1$-forms. The following lemma is essentially an adaptation of results of Mitsumatsu~\cite{M95} and Hozoori~\cite{H22a, H22b}.

\begin{lem} \label{lem:anosovforms}
Let $\Phi$ be a smooth, non-singular flow on $M$ generated by a vector field $X$, then
\begin{enumerate}[label=(\arabic*)]
\item $\Phi$ is oriented projectively Anosov if and only if there exist $(\alpha_s, \alpha_u) \in \Omega^1_{X,0} \times \Omega^1_{X,0}$ and continuous functions $r_u, r_s : M \rightarrow \R$ such that
	\begin{align*}
	\overline{\alpha}_s \wedge \overline{\alpha}_u > 0, & & \mathcal{L}_X \alpha_s = r_s \, \alpha_s, & &\mathcal{L}_X \alpha_u = r_u \, \alpha_u,
	\end{align*}
	and $r_s < r_u$. Here, $\overline{\alpha}_s$ and $ \overline{\alpha}_u$ denote the $1$-forms on $N_X$ induced by $\alpha_s$ and $\alpha_u$, respectively.
\item $\Phi$ is oriented Anosov if and only if there exist $\alpha_s$, $\alpha_u$, $r_s$, $r_u$ as above such that $r_s < 0 < r_u$.
\item $\Phi$ is oriented volume preserving Anosov if and only if there exist $\alpha_s$, $\alpha_u$, $r_s$, $r_u$ as above such that $r_u + r_s = 0$.
\end{enumerate}
Moreover, $\ker \alpha_u = E^{ws}$ and $\ker \alpha_s = E^{wu}$.
\end{lem}

\begin{proof} (1) essentially follows from~\cite[Proposition 3.15]{H22a}. We recall the main arguments. If $\Phi$ is projectively Anosov, we can choose an adapted metric and unit vector fields $\overline{e}_s$ and $\overline{e}_u \in N_X$ of class $\mathcal{C}^0_X$ such that $\overline{e}_s$ spans $\overline{E}^s$, $\overline{e}_u$ spans $\overline{E}^u$ and $(\overline{e}_s, \overline{e}_u)$ is positively oriented. The inequality $r_s < r_u$ follows from the definition of a dominated splitting and the fact that the metric is adapted. If $(\overline{\alpha}_s,\overline{\alpha}_u)$ denotes the dual basis of $(\overline{e}_s, \overline{e}_u)$, it induces a pair $(\alpha_s, \alpha_u) \in \Omega^1_{X,0} \times \Omega^1_{X,0}$, and the relations $\mathcal{L}_X \overline{e}_s = - r_s \overline{e}_s$ and $\mathcal{L}_X \overline{e}_u = - r_u \overline{e}_u$ imply the desired relations for $\alpha_s$ and $\alpha_u$. Reciprocally, if $(\alpha_s, \alpha_u)$ is such a pair, we define $(\overline{e}_s, \overline{e}_u)$ as the dual basis of $(\overline{\alpha}_s,\overline{\alpha}_u)$ and we easily check that it yields a projectively Anosov splitting of $N_X$. This is essentially because for $\star \in \{s,u\}$, $\phi^T_* \overline{e}_{\star} = \exp \left( \int^T_0 r_{\star} \circ \phi^t \, dt \right) \overline{e}_{\star}$ and $r_s < r_u$, where $\Phi = \{\phi^t\}$. In particular, we have $\ker \alpha_s = E^{wu}$ and $\ker \alpha_u = E^{ws}$ since these bundles are uniquely determined by the flow.

(2) follows from (1) and~\cite[Proposition 3.17]{H22a}.

The forward direction of (3) follows from the proof of~\cite[Theorem 3]{M95}. Indeed, assuming that $\Phi$ is volume preserving Anosov, we can arrange that $r_u + r_s = 0$ in the following way. If $\mathrm{dvol}$ is a (smooth) volume form preserved by $\Phi$, then $\tau \coloneqq \iota_X \mathrm{dvol}$ is a non-degenerate $2$-form on $N_X$ invariant under $\Phi$. There exists an adapted metric $g$ of class $\mathcal{C}^0_X$ for which the Anosov splitting is orthogonal and the volume form for the induced metric $\overline{g}$ on $N_X$ is precisely $\tau$.\footnote{The induced metric $\overline{g}$ is the pushforward of the restriction of $g$ to $E^s \oplus E^u$ along the projection $E^s \oplus E^u \rightarrow N_X$, which is an isomorphism. Concretely, if $\overline{v}_1, \overline{v}_2$ are vectors in $N_X$ with lifts $v_1, v_2 \in E^s \oplus E^u$, then $\overline{g}(\overline{v}_1, \overline{v}_2) = g(v_1, v_2)$. Since $X$ is orthogonal to $E^s \oplus E^u$ for $g$, the latter quantity does not depend on the choice of such lifts.} Hence, if $e_s$ and $e_u$ are $\mathcal{C}^0_X$ unit vector fields spanning $E^s$ and $E^u$, respectively, then $\tau (\overline{e}_s, \overline{e}_u) = 1$. Differentiating this equality along $X$ yields $r_u + r_s = 0$ as desired, and we obtain $(\alpha_s, \alpha_u)$ by dualization as before. For the reverse direction, $r_s < 0 < r_u$ since $r_s < r_u$ and $r_s = -r_u$, so $\Phi$ is Anosov by (2). Moreover, if $\theta$ is a smooth $1$-form satisfying $\theta(X) \equiv 1$, then $\mathrm{dvol} \coloneqq \alpha_s \wedge \alpha_u \wedge \theta$ is a $\mathcal{C}^0_X$ volume form preserved by $X$ and by~\cite[Corollary 2.1]{LMM}, this volume form is smooth.
\end{proof}

It is well-known that in dimension three, the regularity of the weak-stable and weak-unstable bundles of a smooth (even $\mathcal{C}^2$) Anosov flow are $\mathcal{C}^1$. We have the following

\begin{lem} \label{lem:c1}
If $\Phi$ is Anosov and smooth, we can further assume that $\alpha_s$, $\alpha_u$, $r_s$ and $r_u$ as in Lemma~\ref{lem:anosovforms} are $\mathcal{C}^1$, i.e., $\alpha_s$ and $\alpha_u$ are $\mathcal{C}^1_X$.
\end{lem}

\begin{proof}
By~\cite[Corollary 1.8]{H94}, $E^{ws}$ and $E^{wu}$ are $\mathcal{C}^1$ and an adapted metric can always be assumed to be smooth, so the construction in Lemma~\ref{lem:anosovforms} yields $\mathcal{C}^1$ $1$-forms $\alpha_s$ and $\alpha_u$. The $\mathcal{C}^1$ regularity of $r_s$ and $r_u$ follows from a trick of Simi\'{c}~\cite{S97}. First, let us choose a $\mathcal{C}^1$ $1$-form $\alpha_u$ such that $\ker \alpha_u = E^{ws}$ and $\mathcal{L}_X \alpha_u  = r_u \, \alpha_u$, where $r_u$ is continuous and positive. Fix a smooth vector field $Z$ positively transverse to $E^{ws}$, so that $f \coloneqq \alpha_u(Z) > 0$. Here, $f$ is $\mathcal{C}^1$ and can be approximated by a smooth function $\widetilde{f} > 0$ so that $h \coloneqq\frac{\widetilde{f}}{f}$ is $\mathcal{C}^1$-close to $1$. Setting $\widetilde{\alpha}_u \coloneqq h  \alpha_u$, $\widetilde{\alpha}_u$ is $\mathcal{C}^1$ and satisfies $\mathcal{L}_X \widetilde{\alpha}_u = \widetilde{r}_u \, \widetilde{\alpha}_u$, where $\widetilde{r}_u$ is $\mathcal{C}^0$-close to $r_u$ and can be assumed to be positive. We now show that $\widetilde{r}_u$ is $\mathcal{C}^1$. Indeed, $\widetilde{\alpha}_u(Z) = \widetilde{f}$ and
\begin{align*}
\widetilde{r}_u \widetilde{f} &= \left(\mathcal{L}_X \widetilde{\alpha}_u \right)(Z) \\
			&= \mathcal{L}_X \left( \widetilde{\alpha}_u(Z)\right) - \widetilde{\alpha}_u \left( \mathcal{L}_X Z\right) \\
			&= X\cdot \widetilde{f}- \widetilde{\alpha}_u \left( \mathcal{L}_X Z\right),
\end{align*}
and the last quantity is $\mathcal{C}^1$ since $X\cdot \widetilde{f}$ and $\mathcal{L}_X Z$ are smooth and $\widetilde{\alpha}_u$ is $\mathcal{C}^1$. The same argument applies to $\alpha_s$.
\end{proof}

\begin{rem} The proof actually shows more. Since $\widetilde{r}_u = u + \widetilde{\alpha}_u (V)$ for some smooth function $u$ and some smooth vector field $V$, and $\mathcal{L}_X \widetilde{\alpha}_u = \widetilde{r}_u \, \widetilde{\alpha}_u$, an immediate induction argument shows that for every integer $n \geq 0$, $\mathcal{L}_X^n \widetilde{\alpha}_u$ and $\mathcal{L}_X^n \widetilde{r}_u$ exist and are $\mathcal{C}^1$, where $\mathcal{L}_X^n \coloneqq \mathcal{L}_X \circ \dots \circ \mathcal{L}_X$ denotes the Lie derivative along $X$ iterated $n$ times. In fact, it is well known that in our setting, the individual leaves of the weak-(un)stable foliation are smooth (see~\cite[Lemma 2.1]{LMM}).
\end{rem}

\begin{rem} The same argument works for smooth projectively Anosov flow whose weak-stable and weak-unstable distributions are $\mathcal{C}^1$. However, there are known examples of smooth projectively Anosov flows in dimension three whose weak distributions are not $\mathcal{C}^1$; see~\cite[Example 2.2.9]{ET}.
\end{rem}

We call a pair of $1$-forms $(\alpha_s, \alpha_u)$ as in Lemma~\ref{lem:anosovforms} $(1)$ (resp.~$(2)$, $(3)$) a \textbf{defining pair} for the projectively Anosov (resp.~Anosov, volume preserving Anosov) flow $\Phi$. We further require defining pairs for (volume preserving) Anosov flows to be $\mathcal{C}^1$. We also impose the following conditions on orientations:
\begin{itemize}
    \item The orientation on $E^{ws}$, induced by the orientation of $X$ and the orientation on $E^s$ or $\overline{E}^s$ which implicitly comes with $\Phi$, agrees with the one induced by $\alpha_u$,
    \item The orientation on $E^{wu}$, induced by the orientation of $X$ and the orientation on $E^u$ or $\overline{E}^u$ which implicitly comes with $\Phi$, agrees with the one induced by $- \alpha_s$.\footnote{The somewhat strange minus sign is explained by the following remark. In the Euclidean plane $\R^2$ with its standard oriented basis $(e_1, e_2)$, the dual basis $(e^*_1, e^*_2)$ induces coorientations on the $x$ and $y$-axis corresponding to the natural orientation on the $x$-axis and to the \emph{opposite} of the natural orientation on the $y$-axis.}
\end{itemize}
Concretely, these properties mean that if $\star \in \{s,u\}$ and $\overline{e}_\star \in \overline{E}^\star$ forms an oriented basis, then $\alpha_\star(\overline{e}_\star) > 0$.

We denote by $\mathcal{D}_\Phi$ the space of defining pairs for $\Phi$ endowed with the $\mathcal{C}^0_X$ topology in the projectively Anosov case, and with the $\mathcal{C}^1$ topology in the (volume preserving) Anosov case.

\begin{lem} \label{lem:contract} The space $\mathcal{D}_\Phi$ of defining pairs for a projectively Anosov (resp.~Anosov, volume preserving Anosov) flow $\Phi$  with its corresponding topology is contractible.
\end{lem}

 \begin{proof}
Let us fix a defining pair $(\alpha_s, \alpha_u) \in \mathcal{D}_\Phi$ for a projectively Anosov flow $\Phi$ generated by a vector field $X$. If $(\alpha'_s, \alpha'_u) \in \mathcal{D}_\Phi$ is any other defining pair, then $\ker \alpha_u = \ker \alpha'_u$ and $\ker \alpha_s = \ker \alpha'_s$, and the orientations on these spaces agree. Hence, there exist (unique) functions $\rho_s, \rho_u : M \rightarrow \R$ of class $\mathcal{C}^0_X$ such that $\alpha'_u = e^{\rho_u} \alpha_u$ and $\alpha'_s = e^{\rho_s} \alpha_s$, and they satisfy
\begin{align} \label{eq:rho1}
r_u' - r_s' = X \cdot (\rho_u - \rho_s) + r_u - r_s > 0.
\end{align}
Here, $r_u'$ and $r_s'$ are such that $\mathcal{L}_X \alpha'_s = r'_s \, \alpha'_s$ and $\mathcal{L}_X \alpha'_u = r'_u \, \alpha'_u$.

Reciprocally, if $\rho_s, \rho_u : M \rightarrow \R$ are functions as above satisfying~\eqref{eq:rho1}, then $(\alpha'_s, \alpha'_u) \coloneqq (e^{\rho_s} \alpha_s, e^{\rho_u} \alpha_u)$ is also a defining pair for $\Phi$.

It follows that $\mathcal{D}_\Phi$ is homeomorphic to 
\begin{align*}
\mathcal{R} \coloneqq \big\{ (\rho_s, \rho_u) \in \mathcal{C}^0_X \times \mathcal{C}^0_X : X \cdot (\rho_u - \rho_s) + r_u - r_s > 0 \big\},
\end{align*}
and $\mathcal{R}$ is obviously convex, hence contractible. The proof for Anosov flows and volume preserving Anosov flows is similar. The condition on $\rho_s$ and $\rho_u$ becomes
\begin{align} \label{eq:rho2}
X \cdot \rho_u + r_u > 0 \qquad and \qquad X \cdot \rho_s + r_s < 0
\end{align}
if $\Phi$ is Anosov, and
\begin{align} \label{eq:rho3}
X \cdot (\rho_u - \rho_s) + r_u - r_s > 0  \qquad and \qquad X \cdot (\rho_u + \rho_s) = 0
\end{align}
if $\Phi$ is volume preserving Anosov. Both conditions~\eqref{eq:rho2} and~\eqref{eq:rho3} are convex in $(\rho_s, \rho_u)$.
\end{proof}

\begin{rem} If $\Phi$ is a (projectively, volume preserving) Anosov flow generated by a vector field $X$, $f : M \rightarrow \R_{>0}$ is a positive function and $\Phi'$ is the (projectively, volume preserving) Anosov flow generated by $fX$, then $\mathcal{D}_\Phi = \mathcal{D}_{\Phi'}$. Indeed, if $(\alpha_s, \alpha_u) \in \mathcal{D}_\Phi$ then for $\star \in \{s,u\}$,
$$\mathcal{L}_{fX} \alpha_\star = f r_\star \, \alpha_\star.$$
Since the stable/unstable bundles of $\Phi'$ are the same as the ones of $\Phi$, $(\alpha_s, \alpha_u)$ is a defining pair for $\Phi'$ with expansion rates $r'_{s,u} = f r_{s,u}$.

Therefore, there is a well-defined notion of defining pairs for (projectively, volume preserving) oriented Anosov \emph{line distributions}.
\end{rem}

	   \subsection{From Anosov flows to Anosov Liouville pairs} \label{sec:std}
	
Throughout this section, we assume that $\Phi$ is a smooth Anosov flow on $M$ and we construct an AL pair supporting $\Phi$, proving the first part of Theorem~\ref{thmintro:supporting}. We choose a $\mathcal{C}^1$ defining pair $(\alpha_s, \alpha_u) \in \mathcal{D}_\Phi$ as in Lemma~\ref{lem:anosovforms}(2). Following~\cite[Section 4]{H22a}, we define
\begin{align} \label{eq:standard}
    \alpha_- \coloneqq \alpha_u +\alpha_s, \qquad \alpha_+ \coloneqq  \alpha_u - \alpha_s.
\end{align}
Note that $\alpha_\pm$ is of class $\mathcal{C}_X^1$, and the orientation compatibility conditions of Definition~\ref{def:bicont} are satisfied. Let $\mathrm{dvol}$ be the $\mathcal{C}^1$ volume form on $M$ defined by $\alpha_s \wedge \alpha_u = \iota_X \mathrm{dvol}$. We will make use of the elementary identities:
\begin{align*}
\alpha_s \wedge d\alpha_u &= -r_u \, \mathrm{dvol},\\
\alpha_u \wedge d\alpha_s &= r_s \, \mathrm{dvol}, \\
\alpha_s \wedge d\alpha_s &= 0,\\
\alpha_u \wedge d\alpha_u &= 0.
\end{align*}
The first one follows from:
\begin{align*}
    \iota_X\big( \alpha_s \wedge d\alpha_u \big)&= - \alpha_s \wedge \iota_X d\alpha_u = - \alpha_s \wedge \mathcal{L}_X \alpha_u = -r_u \, \alpha_s \wedge \alpha_u = -r_u \, \mathrm{dvol},
\end{align*}
and the three others can be obtained by similar computations. We easily deduce:
\begin{align*}
\alpha_+ \wedge d\alpha_+ &= (r_u - r_s) \, \mathrm{dvol}, \\
\alpha_- \wedge d\alpha_- &= - (r_u - r_s) \, \mathrm{dvol}, \\
d(\alpha_- \wedge \alpha_+) &= 2 (r_u + r_s) \, \mathrm{dvol}.
\end{align*}
Since $r_s < 0 < r_u$, $\alpha_-$ and $\alpha_+$ are contact forms and the criterion of Lemma~\ref{lem:charanosov} is satisfied.\footnote{Note that $\lambda \coloneqq  e^{-s} \alpha_- + e^s \alpha_+$ is a Liouville form if and only if $r_u > 0$.} Therefore, $(\alpha_-, \alpha_+)$ is a $\mathcal{C}^1$ AL pair supporting $\Phi$. 

\begin{defn} A \textbf{standard AL pair} supporting $\Phi$ is a $\mathcal{C}^1$ AL pair obtained by the previous construction. We denote by $\mathcal{AL}^\mathrm{std}_\Phi$ the space of these AL pairs, endowed with the $\mathcal{C}^1$ topology.
\end{defn}

There is an obvious (linear) homeomorphism between $\mathcal{D}_\Phi$ and $\mathcal{AL}^\mathrm{std}_\Phi$ induced by the map $(s,u) \mapsto (u+s, u-s)$. Since $\mathcal{D}_\Phi$ is contractible by Lemma~\ref{lem:contract}, $\mathcal{AL}^\mathrm{std}_\Phi$ is contractible as well.

A standard AL pair $(\alpha_-, \alpha_+)$ is not necessarily smooth by definition; it is smooth exactly when the weak-stable and weak-unstable foliations of $\Phi$ are smooth, which is a quite restrictive situation.\footnote{By~\cite[Theorem 4.7]{Gh93}, the smoothness of the weak-(un)stable implies that $\Phi$ is topologically equivalent to an \emph{algebraic Anosov flow}, i.e., the suspension of an Anosov diffeomorphism of the 2-torus, or the geodesic flow on a closed hyperbolic surface, up to finite cover.} Nevertheless, any pair of smooth $1$-forms $(\alpha_-', \alpha_+')$ sufficiently $\mathcal{C}^1$-close to $(\alpha_-, \alpha_+)$ and satisfying $\alpha_\pm'(X) = 0$ is a smooth AL pair supporting $\Phi$. This shows the forward implication in Theorem~\ref{thmintro:supporting}.

Let $R_\pm$ denote the Reeb vector fields of $\alpha_\pm$, defined by
\begin{align*}
\alpha_\pm(R_\pm) &= 1, \\
d\alpha_\pm(R_\pm, \cdot \, ) &= 0.
\end{align*}
Rewriting these equations in terms of $\alpha_s$ and $\alpha_u$, and using the equalities
\begin{align*}
\iota_X d\alpha_s &= \mathcal{L}_X \alpha_s = r_s \, \alpha_s, \\
\iota_X d\alpha_u &= \mathcal{L}_X \alpha_u = r_u \, \alpha_u, \\
\end{align*}
one easily computes:
\begin{align*}
\alpha_s(R_-) &= \frac{r_u}{r_u - r_s}>0, & \alpha_u(R_-) &= \frac{-r_s}{r_u - r_s}>0, \\
\alpha_s(R_+) &= \frac{-r_u}{r_u - r_s}<0, & \alpha_u(R_+) &= \frac{-r_s}{r_u - r_s}>0.
\end{align*}
Therefore,
\begin{itemize}
    \item $R_-$ is positively transverse to $E^{ws}$ and $E^{wu}$,
    \item $R_+$ is positively transverse to $E^{ws}$ and negatively transverse to $E^{wu}$,
\end{itemize}
and this remains true for a smoothing of $(\alpha_-, \alpha_+)$ as above.\footnote{These transversality properties for the Reeb vector fields are a key feature of Anosov flows and are not satisfied for projectively Anosov flows which are not Anosov, see~\cite[Theorem 6.3]{H22a}.} Since $\mathcal{F}^{ws}$ is a \emph{taut} foliation, we obtain that $R_\pm$ has no contractible closed Reeb orbit, thus $\xi_\pm = \ker \alpha_\pm$ is \emph{hypertight}. This was already observed by Hozoori~\cite[Theorem 1.11]{H22a}.

		\subsection{From Anosov Liouville pairs to Anosov flows} \label{sec:AL}
  
We now turn to the second part of the proof of Theorem~\ref{thmintro:supporting}. Let us assume that $\Phi$ is a smooth non-singular flow on $M$ generated by a vector field $X$ and suppose that it is supported by an AL pair $(\alpha_-, \alpha_+)$. By Proposition~\ref{prop:anoliouv}, $(\alpha_-, \alpha_+)$ defines a bi-contact structure $(\xi_-, \xi_+)$ supporting $X$, so $\Phi$ is projectively Anosov and there exists a dominated splitting $N_X \cong \overline{E}^s \oplus \overline{E}^u$ as in Definition~\ref{definition:defanosov}. We shall construct a defining pair $(\alpha_s, \alpha_u)$ as in Lemma~\ref{lem:anosovforms}(2), implying that $\Phi$ is Anosov.

The proof of~\cite[Proposition 2.2.6]{ET} (see also~\cite[Remark 3.10]{H22a}) shows that $\xi_\pm$ is everywhere transverse to $E^{ws}$ and $E^{wu}$. By our orientation conventions, there exist two continuous functions $\sigma_s, \sigma_u :  M \rightarrow \R$ such that
\begin{align*}
    \ker\big\{e^{-\sigma_u} \alpha_- + e^{\sigma_u} \alpha_+ \big\} &= E^{ws}, \\
    \ker\big\{e^{-\sigma_s} \alpha_- -e^{\sigma_s} \alpha_+\big\} &= E^{wu}.
\end{align*}
Note that $\sigma_u$ and $\sigma_s$ are also continuously differentiable along $X$.\footnote{We cannot assume that they are $\mathcal{C}^1$ yet, since we do not know that $\Phi$ is Anosov!} Indeed, if $\overline{e}_s$ is any vector field of class $\mathcal{C}^0_X$ spanning $\overline{E}^s \subset N_X$, then $$e^{\sigma_u} \alpha_+(\overline{e}_s) + e^{-\sigma_u}\alpha_-(\overline{e}_s)=0,$$ hence $$\sigma_u = \frac{1}{2} \ln \left(-\frac{\alpha_-(\overline{e}_s)}{\alpha_+(\overline{e}_s)} \right),$$ and this quantity is continuously differentiable along $X$; the same argument applies to $\sigma_s$.

We define\footnote{The seemingly strange conformal factors will greatly simplify some computations later, in particular the inequality~\eqref{ineq:al2} in the proof of Lemma~\ref{lem:defret}.}
\begin{align*}
    \alpha_u &\coloneqq \frac{1}{2 \sqrt{\cosh(\sigma_u - \sigma_s)}} \big(e^{-\sigma_u} \alpha_- + e^{\sigma_u} \alpha_+ \big), \\
    \alpha_s &\coloneqq \frac{1}{2 \sqrt{\cosh(\sigma_u - \sigma_s)}} \big(e^{-\sigma_s} \alpha_- -e^{\sigma_s} \alpha_+\big),
\end{align*}
so that
\begin{align*}
    \ker \alpha_u &= E^{ws},\\
    \ker \alpha_s &= E^{wu},
\end{align*}
and
\begin{align}
    \alpha_- &= \frac{1}{\sqrt{\cosh(\sigma_u-\sigma_s)}}\big(e^{\sigma_s} \alpha_u + e^{\sigma_u} \alpha_s\big), \label{eq:alpha-} \\
    \alpha_+ &= \frac{1}{\sqrt{\cosh(\sigma_u-\sigma_s)}}\big(e^{-\sigma_s} \alpha_u - e^{-\sigma_u} \alpha_s\big). \label{eq:alpha+}
\end{align}
Note that $\alpha_u$ and $\alpha_s$ are continuously differentiable along $X$, and since $E^{ws}$ and $E^{wu}$ are invariant under $\Phi$, there exist continuous functions $r_s, r_u : M \rightarrow \R$ such that 
\begin{align*}
\mathcal{L}_X \alpha_s &= r_s \, \alpha_s,\\
\mathcal{L}_X \alpha_u &= r_u \, \alpha_u.
\end{align*}
Moreover, 
\begin{align*}
\alpha_- \wedge \alpha_+ = 2 \, \alpha_s \wedge \alpha_u,
\end{align*}
so $\overline{\alpha}_s \wedge \overline{\alpha}_u > 0$. We are left to show that $r_s < 0 < r_u$, which will follow from Lemma~\ref{lem:charanosov}. Let $\mathrm{dvol}$ be the unique volume form on $M$ such that $\alpha_s \wedge \alpha_u = \iota_X \, \mathrm{dvol}=: \tau$.

\begin{lem} \label{lem:approx2}
With the same notations as in Lemma~\ref{lem:charanosov}, we have
\begin{align}
    f_+ &= \frac{e^{-(\sigma_s + \sigma_u)}}{\cosh(\sigma_u-\sigma_s)} \left( X\cdot(\sigma_u-\sigma_s) + r_u - r_s\right), \label{eqsigma1}\\
    f_- &= \frac{e^{(\sigma_s + \sigma_u)}}{\cosh(\sigma_u-\sigma_s)} \left( -X\cdot(\sigma_u-\sigma_s) + r_u - r_s\right), \label{eqsigma2}\\
    f_0 &= 2 (r_u+r_s). \label{eqsigma3}
\end{align}
\end{lem}

\begin{proof} Although $\alpha_\pm$ are smooth, the quantities $\alpha_s$, $\alpha_u$, $\sigma_s$ and $\sigma_u$ are not $\mathcal{C}^1$ so we cannot compute $d\alpha_\pm$ directly by differentiating from~\eqref{eq:alpha-} and~\eqref{eq:alpha+}. However, these quantities are differentiable along $X$ and the functions $f_0$, $f_-$ and $f_+$ can be computed from 
\begin{align*}
\alpha_+ \wedge \mathcal{L}_X \alpha_+ &= - f_+ \, \tau,\\
\alpha_- \wedge \mathcal{L}_X \alpha_- &= f_- \, \tau, \\
\mathcal{L}_X \alpha_- \wedge \alpha_+ + \alpha_- \wedge \mathcal{L}_X \alpha_+ &= f_0 \, \tau.
\end{align*}
Moreover, the quantities $\mathcal{L}_X \alpha_\pm$ \emph{can} be computed from~\eqref{eq:alpha-} and~\eqref{eq:alpha+} by differentiating along $X$. The calculations are left to the reader.
\end{proof}

Since $f_\pm >0$,~\eqref{eqsigma1} and~\eqref{eqsigma2} imply $$0 \leq \vert X\cdot \sigma \vert < r_u - r_s,$$
where $\sigma \coloneqq \sigma_u - \sigma_s$, and the inequality $f_0^2 < 4f_- f_+$ gives
$$(r_u+r_s)^2 \leq \cosh^2(\sigma)(r_u+r_s)^2 < (r_u - r_s)^2 - \big(X \cdot \sigma \big)^2 \leq (r_u-r_s)^2,$$
yielding $r_s < 0 < r_u$ as desired. This concludes the proof of Theorem~\ref{thmintro:supporting}.

\begin{rem} \label{rem:transversepair}
Similar computations (and Lemma~\ref{lem:smooth2}) show that if $\Phi$ is a non-degenerate flow on $M$, the following are equivalent:
\begin{enumerate}[label=(\arabic*)]
\item $\Phi$ is supported by a transverse Liouville pair $(\alpha_-, \alpha_+)$,
\item $\Phi$ is projectively Anosov and admits a defining pair $(\alpha_s, \alpha_u)$ with $r_u > 0$.
\end{enumerate}
Note that in case, the Reeb vector fields for the standard construction of Section~\ref{sec:std} are still transverse to the weak-unstable bundle of $\Phi$, but are not necessarily transverse to the weak-stable bundle of $\Phi$. We wish to call $\Phi$ a \textbf{semi-Anosov flow}. Our techniques would also show that the space of semi-Anosov flows on $M$ is homotopy equivalent to the space of transverse Liouville pairs on $M$.
\end{rem}

		\subsection{Volume preserving Anosov flows}
		
Volume preserving Anosov flows, i.e., Anosov flows preserving a volume form, constitute a remarkable class of Anosov flows. They are topologically transitive, in the sense that they admit a dense orbit. A deep theorem of Asaoka~\cite{A08} implies that on closed $3$-manifold, every transitive Anosov flow is topologically equivalent to a volume preserving one. In this section, we show some striking connections between volume preserving Anosov flows and Anosov Liouville pairs.

\begin{prop} \label{prop:volume}
Let $\Phi$ be a smooth non-singular flow on $M$. Then $\Phi$ is a volume preserving Anosov flow if and only if it is supported by a closed AL pair.
\end{prop}

\begin{proof}
Let us first assume that $\Phi$ preserves a (smooth) volume form $\mathrm{dvol}$, and let $\tau \coloneqq \iota_X \mathrm{dvol}$. Note that $\tau$ is \emph{closed}. Let $(\xi_-, \xi_+)$ be any bi-contact structure supporting $\Phi$, and $\alpha_\pm$ two contact forms such that $\ker \alpha_\pm = \xi_\pm$. There exists a smooth positive function $\kappa : M \rightarrow \R_{>0}$ such that 
$$\alpha_- \wedge \alpha_+ = \kappa \, \tau.$$ 
The positivity of $\kappa$ follows from our conventions on the coorientations of bi-contact structures. Then, $\left(\alpha_-, \frac{1}{\kappa} \alpha_+\right)$ is a closed pair as in Definition~\ref{def:closed}, and it is automatically an AL pair in view of Lemma~\ref{lem:charanosov} since the corresponding function $f_0$ vanishes.

Let us now assume that $\Phi$ is supported by a closed AL pair $(\alpha_-, \alpha_+)$. By Theorem~\ref{thmintro:supporting}, $\Phi$ is Anosov. Let $\theta$ be any smooth $1$-form on $M$ satisfying $\theta(X) \equiv 1$, where $X$ is the vector field generating $\Phi$, and define $\mathrm{dvol} \coloneqq \alpha_- \wedge \alpha_+ \wedge \theta$. It is easy to check that it is a volume form, and
\begin{align*}
\mathcal{L}_X \mathrm{dvol} &= \mathcal{L}_X(\alpha_- \wedge \alpha_+) \wedge \theta + \alpha_- \wedge \alpha_+ \wedge \mathcal{L}_X \theta \\
&=\alpha_- \wedge \alpha_+ \wedge d\theta(X, \cdot \,) \\
&=0,
\end{align*}
hence $\Phi$ preserves a smooth volume form.
\end{proof}

\begin{rem} A special class of closed AL pairs is given by Geiges pairs, defined in~\cite[Section 8.5]{MNW} as pairs of contact forms $(\alpha_-, \alpha_+)$ satisfying 
\begin{align*}
\alpha_+ \wedge d\alpha_+ &= - \alpha_- \wedge d \alpha_- > 0, \\
\alpha_+ \wedge d\alpha_- &= \alpha_- \wedge d\alpha_+ = 0.
\end{align*}
Geiges pairs are called $(-1)$-Cartan structures in~\cite{H22b}, and they are shown to be in correspondence with volume preserving Anosov flows. Here, we note that $(\alpha_-, \alpha_+)$ is a $\mathcal{C}^1$ Geiges pair if and only if $(\alpha_- - \alpha_+, \alpha_- + \alpha_+)$ is a defining pair for the underlying volume preserving Anosov flow. As a result, the space of Geiges pairs supporting a given flow is contractible. Not every (smooth) volume preserving Anosov flow is supported by a smooth (or even $\mathcal{C}^2$) Geiges pair, as it would imply that the weak-stable and weak-unstable bundles are $\mathcal{C}^2$, so the flow would be smoothly equivalent to an algebraic Anosov flow; see~\cite[Th\'eor\`eme A]{G92}.
\end{rem}

The previous proof shows more: for a volume preserving Anosov flow, \emph{any} supporting bi-contact structure can be realized as the kernel of an AL pair. Surprisingly, this is a characteristic feature of volume preserving Anosov flows.

\begin{thm} \label{thm:supporting}
Let $\Phi$ be a smooth Anosov flow on $M$. Then $\Phi$ preserves a volume form if and only if for every (smooth) supporting bi-contact structure $(\xi_-, \xi_+)$, there exists an AL pair $(\alpha_-, \alpha_+)$ such that $\xi_\pm = \ker \alpha_\pm$.
\end{thm}

\begin{proof}
The forward direction follows from the first part of the proof of Proposition~\ref{prop:volume}. Let us assume that every (smooth) bi-contact structure $(\xi_-, \xi_+)$ supporting $\Phi$ is defined by a (smooth) AL pair, and let us fix a defining pair $(\alpha_s, \alpha_u)$ for $\Phi$ with associated expansion rates $r_s$ and $r_u$ as in Lemma~\ref{lem:anosovforms}. Let $A>0$ be a positive real number and $\{\alpha^n_u\}_{n \in \N}$ and $\{\alpha^n_s\}_{n\in \N}$ be sequences of smooth $1$-forms converging to $\alpha_u$ and $\alpha_s$, respectively, in the $\mathcal{C}^1$ topology. For every $n \in \N$, we define
\begin{align*}
    \alpha^n_+ &\coloneqq \alpha^n_u - e^{-A} \alpha^n_s, \\
    \alpha^n_- &\coloneqq \alpha^n_u + e^A \alpha^n_s.
\end{align*}
We also let
\begin{align*}
    \alpha_+ &\coloneqq \alpha_u - e^{-A} \alpha_s, \\
    \alpha_- &\coloneqq \alpha_u + e^A \alpha_s.
\end{align*}

Then, for $n$ sufficiently large (depending on $A$), $(\alpha^n_-, \alpha^n_+)$ defines a (smooth) bi-contact structure $(\xi^n_-, \xi^n_+)$ supporting $\Phi$. By assumption, there exists a smooth positive function $f_n : M \rightarrow \R_{>0}$ such that $(\alpha^n_-, f_n \alpha^n_+)$ is an AL pair defining $(\xi^n_-, \xi^n_+)$. Lemma~\ref{lem:charanosov}  will imply the following

\medskip

\textbf{\textit{Claim.}} \textit{For every $\epsilon > 0$, there exists a smooth function $h_\epsilon : M \rightarrow \R$ such that}
\begin{align}
\vert X \cdot h_\epsilon + r_u + r_s \vert \leq \epsilon. \label{eq:cond}
\end{align}

\medskip

Assuming the Claim for now, it follows that if $\theta$ is a smooth $1$-form such that $\theta(X)\equiv 1$, then for every closed orbit $\gamma$ of $X$, $$\int_\gamma (r_u+r_s) \, \theta = 0.$$

If the flow is transitive, a classical theorem of Liv\v{s}ic implies that there exists a continuous function $h :M \rightarrow \R$ which is differentiable along $X$ and satisfies $$X \cdot h +r_u +r_s = 0.$$ Writing $\mathrm{dvol}' \coloneqq e^h \, \mathrm{dvol} = e^h \, \alpha_s \wedge \alpha_u \wedge \theta$, $\mathcal{L}_X \mathrm{dvol}' = 0$ so $\Phi$ preserves a positive continuous measure, and by~\cite[Corollary 2.1]{LMM}, this measure is a smooth volume form.

It turns out that the condition~\eqref{eq:cond} \emph{implies} that the flow is transitive. We have not been able to find a proof of this fact in the literature. We refer to Appendix~\ref{appendixB} for a proof using the theory of Sina\"{i}--Ruelle--Bowen measures.

We now prove the Claim. Let $\epsilon > 0$ and choose $A>0$ such that $$\frac{\sup_{M}(r_u - r_s)}{\cosh(A)} \leq \epsilon/3.$$ Let $e_s$ and $e_u$ be $\mathcal{C}^1$ vector fields satisfying
\begin{align*}
\alpha_s(e_s) = 1, & & \alpha_s(e_u) = 0, \\
\alpha_u(e_s) = 0, & & \alpha_u(e_u) = 1,
\end{align*}
so that $\alpha_s \wedge \alpha_u(e_s, e_u) = \mathrm{dvol}(X, e_s, e_u) = 1$. Since $(\alpha^n_-, f_n \alpha^n_+)$ is an AL pair for $n$ large enough, Lemma~\ref{lem:charanosov} implies the inequality
\begin{align*}
     \left\vert X \cdot \ln f_n + \frac{d(\alpha^n_- \wedge \alpha^n_+)(X,e_s, e_u)}{\alpha^n_- \wedge \alpha^n_+(e_s, e_u)} \right\vert < \frac{2\sqrt{-\left(\alpha^n_- \wedge d\alpha^n_-(X, e_s, e_u)\right) \cdot \left(\alpha^n_+ \wedge d\alpha^n_+(X, e_s, e_u)\right)}}{\vert\alpha^n_- \wedge \alpha^n_+(e_s, e_u)\vert}.
\end{align*}
One computes
\begin{align*}
    \lim_{n \rightarrow \infty}\,  \frac{d(\alpha^n_- \wedge \alpha^n_+)(X,e_s, e_u)}{\alpha^n_- \wedge \alpha^n_+(e_s, e_u)} &= r_u + r_s, \\
    \lim_{n \rightarrow \infty}\,  \frac{2\sqrt{-\left(\alpha^n_- \wedge d\alpha^n_-(X, e_s, e_u)\right) \cdot \left(\alpha^n_+ \wedge d\alpha^n_+(X, e_s, e_u)\right)}}{\vert\alpha^n_- \wedge \alpha^n_+(e_s, e_u)\vert} &= \frac{r_u - r_s}{\cosh(A)}, 
\end{align*}
by first replacing $\alpha^n_\pm$ by $\alpha_\pm$ and then writing these expressions in terms of $\alpha_s$ and $\alpha_u$. We obtain that for $n$ large enough such that the above two sequences are $\epsilon/3$-close to their limits,
\begin{align*}
    \vert X \cdot \ln f_n + r_u + r_s \vert \leq \frac{r_u-r_s}{\cosh(A)} + \epsilon/3 + \epsilon/3 \leq \epsilon,
\end{align*}
hence $h_\epsilon \coloneqq \ln f_n$ satisfies the required inequality and the Claim is proved.
\end{proof}

\begin{rem} The proof can be adapted to show that if every bi-contact structure supporting $\Phi$ is realized as the kernel of a Liouville pair, then the determinant of the Poincar\'e return map for every closed orbit of $\Phi$ is bigger or equal than $1$. We expect that this should also imply that $\Phi$ is volume preserving.
\end{rem}

		\section{Spaces of Anosov Liouville pairs and bi-contact structures} \label{sec:spaces}

This section is dedicated to the proof of Theorem~\ref{thmintro:homeq} from the Introduction. We first describe the main strategy in a more general setting. Let $E$ and $B$ be topological spaces and $f : E \rightarrow B$ be a continuous map. We can assume that $E$ and $B$ have the homotopy type of CW complexes. This is the case for the spaces we consider (e.g., $\mathcal{AL}$, $\mathcal{BC}$, $\mathcal{AF}$, $\mathbb{P}\mathcal{AF}$, etc.) as they are open subsets of Fr\'echet spaces.\footnote{Indeed, Fr\'{e}chet spaces are absolute neighborhood retracts (ANRs) by a theorem of Dugundji; an open subset of an ANR is an ANR; every ANR has the homotopy type of a CW complex by a theorem of Milnor and Whitehead. However, it is known that an open subset of an infinite dimensional Fr\'{e}chet space is \emph{not} a CW complex.} To show that $f$ is a homotopy equivalence, it is enough to show that it is a Serre fibration with (weakly) contractible fibers. However, it seems rather hard to show that the maps we care about (e.g., $\mathcal{I}$, $\mathbb{P}\mathcal{I}$) satisfy a homotopy lifting property, as this would require a careful understanding of how the stable and unstable bundles depend on the (projectively) Anosov flow. Instead, we choose a more indirect approach: we first show that these maps have contractible fibers, and we then show that they are \emph{topological submersions}.

\begin{defn}
    $f : E \rightarrow B$ is a \textbf{topological submersion} if it is surjective, and for every $x\in E$, there exists a neighborhood $U$ of $x$ in $E$ such that if we write $y \coloneqq f(x)$ and $V\coloneqq U \cap f^{-1}(y)$, there exists a homeomorphism $U \overset{\sim}{\longrightarrow} f(U) \times V$ making the following diagram commute:
    $$ \begin{tikzcd}[column sep=small]
    U \arrow[rr, "\sim"] \arrow[dr, "f"']
	&                         & f(U) \times V \arrow[dl, "\mathrm{pr}_1"] \\
& f(U)   & 
\end{tikzcd}$$
Here, $\mathrm{pr}_1$ denotes the projection onto the first factor.
\end{defn}

Fiber bundles are topological submersions, but the converse is not true since the product structure of topological submersions is only \emph{local on the domain} and is not ``uniform in the fibers''; see Figure~\ref{fig:topsub} below. However, we have:

\begin{lem} \label{lem:topfib}
If $f:E \rightarrow B$ is a topological submersion with (weakly) contractible fibers, then $f$ an acyclic Serre fibration.
\end{lem}

\begin{proof}
    Since projections are open and openness is a local property, $f$ is open. By~\cite[Lemma 6]{M02}, $f$ is a \emph{homotopic submersion} (see~\cite[Definition 1]{M02}), also known as a \emph{Serre micro-fibration}~\cite{Gr86}. The result then follows from~\cite[Corollary 13]{M02} (see also~\cite[Lemma 2.2]{W05}).
\end{proof}

\begin{figure}[t]
    \begin{center}
        \begin{picture}(90, 86)(0,0)
        \put(0,0){\includegraphics[width=90mm]{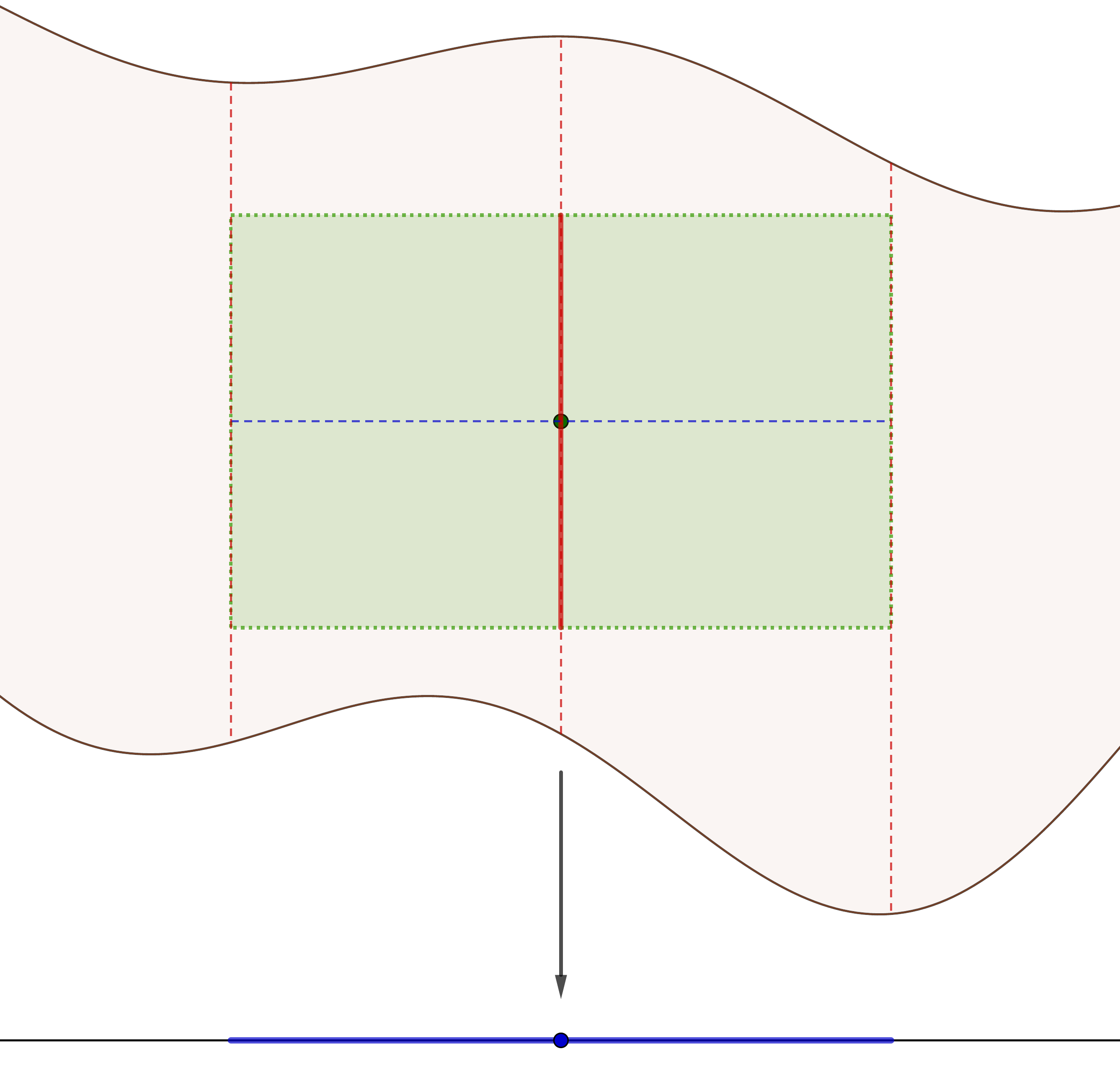}}
        \put(41.5,14){$f$}
        \put(8,5){$B$}
        \put(8,61){$E$}
        \color{blue1}
        \put(27,5){$f(U)$}
        \put(46.5,5){$y$}
        \color{red1}
        \put(46.5,61){$V$}
        \color{green1}
        \put(46.5,54){$x$}
        \put(30,61){$U$}
        \end{picture}
        \caption{A topological submersion.}
        \label{fig:topsub}
    \end{center}
\end{figure}
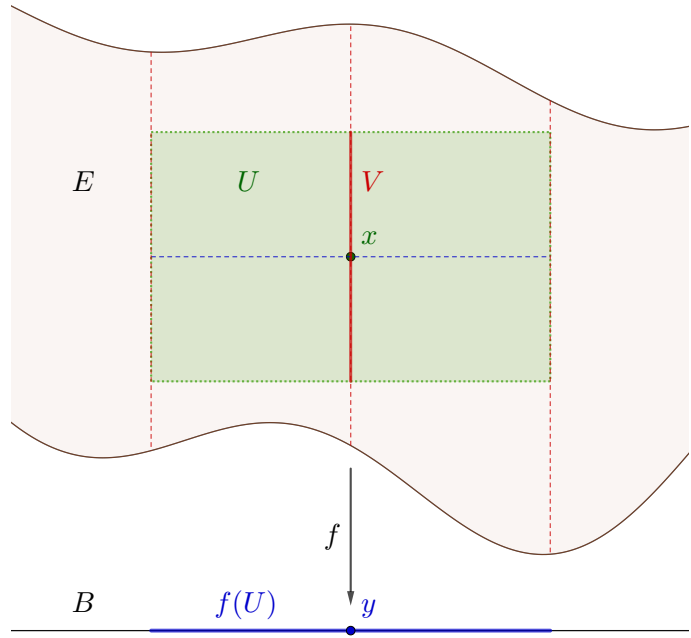

		\subsection{Contractibility of fibers} \label{section:contract}

In this section, we show:

\begin{thm} \label{thm:contract} Let $\Phi$ be a smooth Anosov flow on $M$. The spaces of AL pairs and weak AL pairs supporting $\Phi$ are both contractible.
\end{thm}

We also show a similar result for projectively Anosov flows:

\begin{thm} \label{thm:projcontract} Let $\Phi$ be a smooth projectively Anosov flow on $M$. The space of bi-contact structures supporting $\Phi$ is contractible.
\end{thm}

We obtain a version for volume preserving Anosov flows as well:

\begin{thm} \label{thm:volcontract}
Let $\Phi$ be a smooth volume preserving Anosov flow on $M$. The space of closed AL pairs supporting $\Phi$ is contractible.
\end{thm}

If $\Phi$ is a smooth Anosov flow on $M$,
\begin{itemize}
\item $\mathcal{AL}_\Phi$ denotes the space of smooth AL pairs on $M$ supporting $\Phi$, endowed with the $\mathcal{C}^\infty$ topology,
\item $\mathcal{AL}^1_\Phi$ denotes the space of $\mathcal{C}^1$ AL pairs on $M$ supporting $\Phi$, endowed with the $\mathcal{C}^1$ topology.
\end{itemize}
Smooth AL pairs supporting $\Phi$ form a dense subset of $\mathcal{AL}^1_\Phi$. Recall that  $\mathcal{D}_\Phi$ denotes the space of $\mathcal{C}^1$ defining pairs on $M$ for $\Phi$, and $\mathcal{AL}^{\mathrm{std}}_\Phi \subset \mathcal{AL}^1_\Phi$ denotes the space of $\mathcal{C}^1$ standard AL pairs supporting $\Phi$, both endowed with the $\mathcal{C}^1$ topology. Theorem~\ref{thm:contract} will be a consequence of the contractibility of $\mathcal{AL}^{\mathrm{std}}_\Phi $ and the following two lemmas.

\begin{lem} \label{lem:heq}
The natural map $\mathcal{AL}_\Phi \rightarrow \mathcal{AL}^1_\Phi$ is a homotopy equivalence.
\end{lem}

\begin{proof} This follows from some standard facts in algebraic topology. We will use that homotopy equivalences are \emph{local} in the following sense:

\medskip

\textit{\textbf{Fact.}} \textit{A continuous map $f: X \rightarrow Y$ between topological spaces is a homotopy equivalence if there exists a numerable open cover $\mathcal{U}$ of $Y$ satisfying
\begin{enumerate}[label=(\arabic*)]
\item $\mathcal{U}$ is stable under finite intersections,
\item For every $U \in \mathcal{U}$, $f : f^{-1}(U) \rightarrow U$ is a homotopy equivalence.
\end{enumerate}
}

\medskip

See~\cite[Theorem 1]{D71} for a proof of this fact. Recall that an open cover is \emph{numerable} if it admits a subordinate partition of unity. In our context, covers are automatically numerable since all the spaces under consideration are metrizable. We can simply cover $\mathcal{AL}^1_\Phi$ by sufficiently small open $\mathcal{C}^1$ balls and refine this cover by taking all possible finite intersections. These balls are convex as subsets of the space of pairs of $\mathcal{C}^1$ $1$-forms on $M$, and so are finite intersections thereof, so all of the open subsets in our cover are contractible. Since smooth AL pairs supporting $\Phi$ are dense in $\mathcal{AL}^1_\Phi$, every open $\mathcal{C}^1$ ball in $\mathcal{AL}^1_\Phi$ intersects $\mathcal{AL}_\Phi$. The intersection of such a ball with $\mathcal{AL}_\Phi$ is also convex as a subset of $\Omega^1(M) \times \Omega^1(M)$, and so are finite intersections of such balls with $\mathcal{AL}_\Phi$. Thus, the conditions (1) and (2) of the Fact are trivially satisfied.
\end{proof}

\begin{lem} \label{lem:defret}
$\mathcal{AL}^{\mathrm{std}}_\Phi$ is a strong deformation retract of $\mathcal{AL}^1_\Phi$.\end{lem}

\begin{proof}
Let $(\alpha_-, \alpha_+) \in \mathcal{AL}^1_\Phi$. As in Section~\ref{sec:AL}, there exist functions $\sigma_s, \sigma_u : M \rightarrow \R$ satisfying
\begin{align*}
    \ker\big\{e^{-\sigma_u} \alpha_- + e^{\sigma_u} \alpha_+ \big\} &= E^{ws}, \\
    \ker\big\{e^{-\sigma_s} \alpha_- -e^{\sigma_s} \alpha_+\big\} &= E^{wu}.
\end{align*}
If $e_s$ and $e_u$ are $\mathcal{C}^1$ vector fields such that $\overline{e}_s$ spans $\overline{E}^s$ and $\overline{e}_u$ spans $\overline{E}^u$, we can write
\begin{align*}
\sigma_s = \frac{1}{2} \ln \left(  \frac{\alpha_+(e_u)}{\alpha_-(e_u)} \right), \qquad \sigma_u = \frac{1}{2} \ln \left( - \frac{\alpha_-(e_s)}{\alpha_+(e_s)} \right), 
\end{align*}
so $\sigma_s$ and $\sigma_u$ are $\mathcal{C}^1$. Moreover, the map $\mathcal{S}: (\alpha_-, \alpha_+) \mapsto (\sigma_s, \sigma_u)$ is continuous in the $\mathcal{C}^1$ topology. As before, we define
\begin{align*}
    \alpha_u &\coloneqq \frac{1}{2 \sqrt{\cosh(\sigma_u - \sigma_s)}} \big(e^{-\sigma_u} \alpha_- + e^{\sigma_u} \alpha_+ \big), \\
    \alpha_s &\coloneqq \frac{1}{2 \sqrt{\cosh(\sigma_u - \sigma_s)}} \big(e^{-\sigma_s} \alpha_- -e^{\sigma_s} \alpha_+\big).
\end{align*}
The computations of Section~\ref{sec:AL} show that $(\alpha_s, \alpha_u) \in \mathcal{D}_\Phi$. Therefore, we obtain a continuous map $\mathcal{D}: (\alpha_-, \alpha_+) \mapsto (\alpha_s, \alpha_u)$. We now define a strong deformation retraction $r : \mathcal{AL}^1_\Phi \times [0,1] \rightarrow \mathcal{AL}^1_\Phi$. Let $(\alpha_-, \alpha_+) \in \mathcal{AL}^1_\Phi $, and $(\sigma_s, \sigma_u) = \mathcal{S}(\alpha_-, \alpha_+)$ and $(\alpha_s, \alpha_u) = \mathcal{D} (\alpha_-, \alpha_+)$ as before. For $t \in [0,1]$, we define
\begin{align*}
    \alpha^t_- &\coloneqq \frac{1}{\sqrt{\cosh((1-t)\sigma)}}\big(e^{(1-t)\sigma_s} \alpha_u + e^{(1-t)\sigma_u} \alpha_s\big), \\
    \alpha^t_+ &\coloneqq \frac{1}{\sqrt{\cosh((1-t)\sigma)}}\big(e^{-(1-t)\sigma_s} \alpha_u - e^{-(1-t)\sigma_u} \alpha_s\big),
\end{align*}
where $\sigma = \sigma_u - \sigma_s$. We then have
\begin{itemize}
\item $\big(\alpha^0_-, \alpha^0_+\big) = (\alpha_-, \alpha_+)$,
\item $\big(\alpha^1_-, \alpha^1_+\big) =  \big( \alpha_u + \alpha_s, \alpha_u - \alpha_s \big) \in \mathcal{AL}^\mathrm{std}_\Phi$,
\item If $(\alpha_-, \alpha_+) \in \mathcal{AL}^{\mathrm{std}}_\Phi$, then $\big(\alpha^t_-, \alpha^t_+\big) = (\alpha_-, \alpha_+)$ for every $t \in [0, 1]$.
\end{itemize}
 We claim that $\big(\alpha^t_-, \alpha^t_+\big)$ is an AL pair for every $t \in [0,1]$. Indeed, by Lemma~\ref{lem:charanosov} and Lemma~\ref{lem:approx2}, it is enough to show the inequality
\begin{align} \label{ineq:al2}
\cosh^2 \big((1-t)\sigma \big)(r_u + r_s )^2 + (1-t)^2 \big(X\cdot \sigma\big)^2 < (r_u-r_s)^2.
\end{align}
It holds for $t=0$ since $(\alpha_-, \alpha_+)$ is an AL pair, and the left-hand side is obviously a non-increasing function of $t$, so it holds for every $t \in [0,1]$.

We finally define $r\big((\alpha_-, \alpha_+), t\big) \coloneqq \big(\alpha^t_-, \alpha^t_+\big)$ so that $r : \mathcal{AL}^1_\Phi \times [0,1] \rightarrow \mathcal{AL}^1_\Phi$ is continuous, and by the three bullets above, it is a strong deformation retraction of $\mathcal{AL}^1_\Phi$ onto $\mathcal{AL}^{\mathrm{std}}_\Phi$.
\end{proof}

\begin{proof}[Proof of Theorem~\ref{thm:contract}] For AL pairs, combine Lemma~\ref{lem:contract}, Lemma~\ref{lem:heq} and Lemma~\ref{lem:defret}. For weak AL pairs, the argument can be modified as follows. Lemma~\ref{lem:heq} can be easily adapted to the case of weak AL pairs. Lemma~\ref{lem:defret} can be adapted to show that the space of $\mathcal{C}^1$ weak AL pairs supporting a smooth Anosov flow $\Phi$ deformation retracts onto the space of pairs of the form $\alpha_\pm = \alpha_u \mp \alpha_s$, where $(\alpha_s, \alpha_u)$ satisfies the conditions of a defining pair for $\Phi$ \emph{without the condition $r_s < 0$} (but still satisfies $r_s < r_u$ and $0<r_u$, see Remark~\ref{rem:transversepair}). The latter space is convex, hence contractible.
\end{proof}

\begin{proof}[Proof of Theorem~\ref{thm:projcontract}] We only sketch how to modify the proof of Theorem~\ref{thm:contract} and we leave the details to the interested reader.

First of all, we shall introduce the space of $\mathcal{C}^0_X$ bi-contact structures. Those are continuous pairs of codimension $1$ distributions $(\xi_-, \xi_+)$ which are continuously differentiable along $X$ and which are defined by some $1$-forms $\alpha_-, \alpha_+ \in \Omega^0_X$ satisfying
\begin{align} \label{eq:weakcont}
\overline{\alpha}_- \wedge \overline{\alpha}_+ > 0, & &\overline{\alpha}_- \wedge \mathcal{L}_X \overline{\alpha}_- < 0, & &\overline{\alpha}_+ \wedge \mathcal{L}_X \overline{\alpha}_+ > 0,
\end{align}
as forms on $N_X$. 

For the purpose of the proof, we choose an arbitrary $\mathcal{C}^0_X$ vector field $e_u$ such that $\overline{e}_u$ spans $\overline{E}^u \subset N_X$ and defines the prescribed orientation. This is equivalent to choosing a $1$-form $\alpha_u \in \Omega^1_{X,0}$ such that $\mathrm{ker} \alpha_u = E^{ws}$ as oriented $2$-plane fields, with the normalization $\alpha_u(e_u) \equiv 1$.

We denote by $\mathcal{BC}_\Phi$ (resp.~$\mathcal{BC}^0_\Phi$) the space of smooth (resp.~$\mathcal{C}^0_X$) bi-contact structures supporting $\Phi$. We write $\mathcal{BC}^\mathrm{std}_\Phi \subset \mathcal{BC}^0_\Phi$ for the space of \emph{standard} bi-contact structures supporting $\Phi$, of the form $\big(\ker \big( \alpha_u + \alpha_s \big), \ker \left( \alpha_u - \alpha_s \right) \big)$, where $\alpha_u$ is fixed by the condition $\alpha_u(e_u) \equiv 1$.

Lemma~\ref{lem:heq} can be easily adapted to show that the natural map $\mathcal{BC}_\Phi \rightarrow \mathcal{BC}^0_\Phi$ is a homotopy equivalence, using Lemma~\ref{lem:smooth2}.

Lemma~\ref{lem:defret} can be modified as follows. For $(\xi_-, \xi_+) \in \mathcal{BC}^0_\Phi$, we denote by $(\alpha_-, \alpha_+)$ the unique pair of $1$-forms in $\Omega^0_X$ satisfying $\ker \alpha_\pm = \xi_\pm$ and $\alpha_\pm(e_u) = 1$. We define a $\mathcal{C}^0_X$ function $\sigma : M \rightarrow \R$ by $$\ker\big\{e^{-\sigma} \alpha_- + e^{\sigma} \alpha_+ \big\} = E^{ws},$$  so that $$\alpha_u = \frac{1}{2 \cosh(\sigma)}\big( e^{-\sigma} \alpha_- + e^{\sigma} \alpha_+\big),$$ and we define $$\alpha_s \coloneqq \frac{1}{2 \cosh(\sigma)} \big(\alpha_- - \alpha_+ \big)$$ so that $\alpha_s \in \Omega^1_{X,0}$, and $$\mathrm{ker} \alpha_s = E^{wu}.$$ This readily implies that 
\begin{align*}
\alpha_- &= \alpha_u + e^{\sigma} \alpha_s, \\
\alpha_+ &= \alpha_u - e^{-\sigma} \alpha_s.
\end{align*}
Writing $\mathcal{L}_X \alpha_u = r_u \alpha_u$ and $\mathcal{L}_X \alpha_s = r_s \alpha_s$, where $r_s, r_u : M \rightarrow \R$ are continuous, the last two inequalities in~\eqref{eq:weakcont} are equivalent to $$\vert X \cdot \sigma \vert < r_u - r_s.$$ Moreover, $$\alpha_- \wedge \alpha_+ = 2 \cosh(\sigma) \alpha_s \wedge \alpha_u.$$ This shows that $(\alpha_s, \alpha_u)$ is a defining pair for $\Phi$ that satisfies $\alpha_u(e_u) \equiv 1$. For $t \in [0,1]$, we define
\begin{align*}
\alpha^t_- &\coloneqq \alpha_u + e^{(1-t) \sigma} \alpha_s, & &\xi^t_- \coloneqq \ker \alpha^t_-\\
\alpha^t_+ &\coloneqq \alpha_u - e^{-(1-t) \sigma} \alpha_s, & & \xi^t_+ \coloneqq \ker \alpha^t_+.
\end{align*}
These formulae define a strong deformation retraction of $\mathcal{BC}^0_\Phi$ onto $\mathcal{BC}^\mathrm{std}_\Phi$. Moreover, $\mathcal{BC}^\mathrm{std}_\Phi$ is homeomorphic to the space of defining pairs $(\alpha_s, \alpha_u)$ for $\Phi$ satisfying $\alpha_u(e_u) \equiv 1$, and one easily checks that this space is contractible.
\end{proof}

\begin{proof}[Proof of Theorem~\ref{thm:volcontract}] The result essentially follows from Theorem~\ref{thm:projcontract}. Let $\mathrm{dvol}$ be a smooth volume form preserved by $\Phi$ and $\tau \coloneqq \iota_X \mathrm{dvol}$. If $(\alpha_-, \alpha_+)$ is a closed AL pair supporting $\Phi$, there exists a smooth positive function $\kappa \coloneqq M \rightarrow \R_{>0}$ such that $$\alpha_- \wedge \alpha_+ = \kappa \, \tau.$$ Moreover, $X \cdot \kappa = 0$ and since $\Phi$ is topologically transitive, $\kappa$ is constant. One easily checks that the space of closed AL pairs supporting $\Phi$ is homotopy equivalent to the space of balanced pairs of contact forms $(\alpha_-, \alpha_+)$ supporting $\Phi$ and satisfying $\alpha_- \wedge \alpha_+ = \tau$. We denote this space by $\mathcal{BC}^\tau_\Phi$. There is a natural continuous map $\mathcal{K} : \mathcal{BC}^\tau_\Phi \rightarrow \mathcal{BC}_\Phi$, obtained by taking kernels, which is surjective by Theorem~\ref{thm:supporting}. One easily checks that $\mathcal{K}$ is injective and that it is a homeomorphism. Theorem~\ref{thm:projcontract} finishes the proof.
\end{proof}

		\subsection{Homotopy equivalences}

Let us recall some notations.
\begin{itemize}
\item $\mathcal{AL}$ denotes the space of smooth AL pairs on $M$,
\item $\mathcal{AF}$ denotes the space of smooth Anosov flows on $M$, up to positive time reparametrization,
\item $\mathbb{P} \mathcal{AF}$ denotes the space of smooth projectively Anosov flows on $M$,  up to positive time reparametrization.
\end{itemize}

Recall that there is a continuous map
$$
\begin{array}{rccc}
\mathcal{I} : & \mathcal{AL} & \longrightarrow & \mathcal{AF} \\
		   & (\alpha_-, \alpha_+) & \longmapsto & \ker \alpha_- \cap \ker \alpha_+
\end{array} $$
where we identify an oriented $1$-dimensional distribution with any flow spanned by it. Similarly, there is a continuous map
$$
\begin{array}{rccc}
\mathbb{P}\mathcal{I} : & \mathcal{BC} & \longrightarrow & \mathbb{P}\mathcal{AF} \\
		   & (\xi_-, \xi_+) & \longmapsto & \xi_- \cap \xi_+
\end{array} $$

In this section, we show the main theorems of this article:

\begin{thm} \label{thm:homeq} The map $\mathcal{I}$ is an acyclic Serre fibration.
\end{thm}

Our argument can easily be adapted to the case of projectively Anosov flows (this result might already be known to some experts):

\begin{thm} \label{thm:homeqproj} The map $\mathbb{P} \mathcal{I}$ is an acyclic Serre fibration.
\end{thm}

\begin{rem} With more work, it is possible to show that $\mathcal{I}$ is \textbf{shrinkable}: it is homotopy equivalent over $\mathcal{AF}$ to $\mathrm{id} :  \mathcal{AF} \rightarrow \mathcal{AF}$. Concretely, this means that there exists a section $s$ of $\mathcal{I}$ such that $s \circ \mathcal{I}$ is \emph{fiberwise homotopic} to $\mathrm{id}$. This implies that the space sections of $\mathcal{I}$ is non-empty and contractible. To prove this statement, one would need to upgrade the results of Section~\ref{section:contract} to hold \emph{in family over $\mathcal{AF}$}. A key ingredient is~\cite[Lemma 2.1]{LMM}, which shows that for smooth Anosov flows, the Anosov splitting depends continuously on the flow. We are not aware of a similar result for projectively Anosov flows.
\end{rem}

We will need the following 
\begin{lem} \label{lem:topsub}
$\mathcal{I}$ is a topological submersion.
\end{lem}

\begin{proof} By Theorem~\ref{thmintro:supporting}, $\mathcal{I}$ is surjective. We fix some auxiliary Riemannian metric $g$ on $M$ and identify $\mathcal{AF}$ with the space of unit Anosov vector fields on $M$. Let $\boldsymbol{\alpha}^0 = (\alpha^0_-, \alpha^0_+) \in \mathcal{AL}$ and let $X$ be the unit vector field generating $\mathcal{I}(\boldsymbol{\alpha}^0)$, whose flow is denoted by $\Phi$. We choose an arbitrary smooth $1$-form $\theta$ such that $\theta(X) \equiv 1$. For a unit vector field $X'$ sufficiently close to $X$ (so that $\theta(X') > 0$) and $\boldsymbol{\alpha} = (\alpha_-, \alpha_+) \in \mathcal{AL}_\Phi$, we define
\begin{align*}
\alpha'_\pm &\coloneqq \alpha_\pm - \frac{\alpha_\pm(X')}{\theta(X')} \, \theta,
\end{align*}
so that $\alpha'_\pm(X') = 0$. Since $\mathcal{AL} \subset \Omega^1 \times \Omega^1$ is open, we can find an open neighborhood $\mathcal{N}_{\boldsymbol{\alpha}^0}$ of $\boldsymbol{\alpha}^0$ in $\mathcal{AL}_\Phi$ and an open neighborhood $\mathcal{N}_\Phi$ of $\Phi$ in $\mathcal{AF}$ such that the map
$$\begin{array}{rccc}
\psi : &  \mathcal{N}_\Phi \times \mathcal{N}_{\boldsymbol{\alpha}^0} & \longrightarrow & \mathcal{AL} \\
	& \big(X',\boldsymbol{\alpha} \big) & \longmapsto & \boldsymbol{\alpha}'
\end{array}$$
is well-defined. It is continuous and satisfies $\mathcal{I} \circ \psi = \mathrm{pr}_1$. Moreover, the restriction of $\psi$ to $\{X\} \times \mathcal{N}_{\boldsymbol{\alpha}^0}$ is the inclusion $\mathcal{N}_{\boldsymbol{\alpha}^0} \subset \mathcal{AL}$. One easily checks that $\psi$ is injective, has open image and has an inverse given by $\psi^{-1}(\boldsymbol{\alpha}') \coloneqq \big(X', \boldsymbol{\alpha} \big)$ where
\begin{align*}
\alpha_\pm &\coloneqq \alpha'_\pm - \alpha'_\pm(X) \, \theta,
\end{align*}
and $X'$ is the unit vector field spanning $\mathcal{I}\big( \boldsymbol{\alpha}'\big)$. Therefore, $\psi^{-1}$ is a local trivialization of $\mathcal{I}$ around $\boldsymbol{\alpha}^0$.
\end{proof}

\begin{proof}[Proof of Theorem~\ref{thm:homeq}]
By Theorem~\ref{thm:contract} and Lemma~\ref{lem:topsub}, $\mathcal{I}$ is a topological submersion with contractible fibers, hence an acyclic Serre fibration by Lemma~\ref{lem:topfib}.
\end{proof}

\begin{proof}[Proof of Theorem~\ref{thmintro:louvstr}] Lemma~\ref{lem:topsub} and its proof hold verbatim for $\mathcal{I}^w$ so the previous proof applies to $\mathcal{I}^w$ as well.
\end{proof}

\begin{proof}[Proof of Theorem~\ref{thm:homeqproj}]
By Theorem~\ref{thm:projcontract}, we already know that the fibers of $\mathbb{P}\mathcal{I}$ are contractible so it is enough to adapt Lemma~\ref{lem:topsub}. It can be done by choosing an auxiliary smooth vector field $Z$, depending on an initial choice of $\big(\xi^0_-, \xi^0_+\big) \in \mathcal{BC}$, which is positively transverse to $\xi^0_\pm$ and satisfies $\theta(Z) \equiv 0$. We can uniquely choose contact forms $\alpha^0_\pm$ for $\xi^0_\pm$ by imposing $\alpha_\pm(Z) \equiv 1$. The proof of Lemma~\ref{lem:topsub} can be reproduced with minor changes to provide a suitable trivialization near $\big(\xi^0_-, \xi^0_+\big)$.
\end{proof}

		\subsection{The kernel map}

Recall that we have a continuous map
$$ \begin{array}{rccc}
\underline{\ker}: & \mathcal{AL} & \longrightarrow & \mathcal{BC} \\
& (\alpha_-, \alpha_+) & \longmapsto & (\ker \alpha_-, \ker \alpha_+ ),
\end{array}$$
where the spaces $\mathcal{AL}$ and $\mathcal{BC}$ are endowed with the $\mathcal{C}^\infty$ topology.

\begin{lem} \label{lem:keropen}
The map $\underline{\ker}$ is open.
\end{lem}

\begin{proof}
It easily follows from the openness of $\mathcal{AL}$ and $\mathcal{BC}$ in the space of smooth $1$-forms on $M$ and the space of smooth plane fields on $M$, respectively, and the following elementary fact. If $\mathring{\Omega}^1 \subset \Omega^1$ denotes the space of nowhere vanishing $1$-forms on $M$ and $\Pi$ denotes the space of smooth cooriented plane fields on $M$, the natural map 
$$ \begin{array}{rccc}
\ker : & \mathring{\Omega}^1 & \longrightarrow & \Pi \\
& \alpha & \longmapsto & \ker \alpha
\end{array}$$
is open (for the $\mathcal{C}^\infty$ topology on the domain and target). Indeed, after trivializing the tangent bundle of $M$ and fixing an auxiliary Riemannian metric, we can identify $\Pi$ with the space of smooth maps $M \rightarrow S^3$ (via the unit normal vector) and $\mathring{\Omega}^1$ with the space of smooth maps $M \rightarrow \R^3 \setminus \{0\}$, so that $\mathrm{ker}$ becomes the composition with the standard projection $\R^3 \setminus \{0\} \cong \R \times S^3 \rightarrow S^3$. Ultimately, $\mathrm{ker}$ boils down to the projection $\mathcal{C}^\infty(M, \R) \times \mathcal{C}^\infty(M, S^3) \rightarrow \mathcal{C}^\infty(M, S^3)$ onto the second factor, which is clearly open.
\end{proof}

\begin{thm} \label{thm:hominc}
The map $\underline{\ker}$ is an acyclic Serre fibration onto its image.
\end{thm}

\begin{proof}
As before, by Lemma~\ref{lem:topfib}, it is enough to show the following properties.
\begin{enumerate}[label=(\arabic*)]
\item The fibers of $\underline{\ker}$ over its image are contractible.
\item $\underline{\ker}$ is a topological submersion onto its image.
\end{enumerate}

We can simplify the situation by restricting to the space $\mathcal{AL}^b$ of \emph{balanced} AL pairs, since there is a homeomorphism
$$\begin{array}{rccc} \label{eq:homeoAL} 
\vartheta: & \mathcal{C}^\infty(M, \R) \times \mathcal{AL}^b & \overset{\sim}{\longrightarrow} & \mathcal{AL}\\
& \big(\sigma, (\alpha_-, \alpha_+)\big) & \longmapsto & \big(e^{-\sigma} \alpha_-,  e^\sigma \alpha_+\big),
\end{array}$$
and $\underline{\ker}$ is compatible with this homeomorphism in the obvious way.

Let us consider a bi-contact structure $(\xi_-, \xi_+)$ defined by a balanced AL pair $(\alpha_-, \alpha_+)$, and let $\mathrm{dvol} \coloneqq \alpha_+ \wedge d\alpha_+$. We also consider a vector field $X \in \xi_- \cap \xi_+$ normalized so that $\alpha_- \wedge \alpha_+ = \iota_X \mathrm{dvol}$.

To show (1), note that any other balanced AL pair defining $(\xi_-, \xi_+)$ is of the form $(e^\sigma \alpha_-, e^\sigma \alpha_+)$ for a smooth function $\sigma : M \rightarrow \R$ satisfying $$ \vert 2 X\cdot \sigma + f_0 \vert<2,$$
where we use the notation of Lemma~\ref{lem:charanosov}. By assumption, $\vert f_0 \vert < 2$, so the space of $\sigma$ such that $(e^\sigma \alpha_-, e^\sigma \alpha_+)$ is a balanced AL pair defining $(\xi_-, \xi_+)$ is convex, hence contractible.

To show (2), we consider an open neighborhood $\mathcal{V}$ of $(\alpha_-, \alpha_+)$ in $\mathcal{AL}^b$. We can find a smaller neighborhood $\mathcal{V}' \subset \mathcal{V}$ such that for every $(\alpha'_-, \alpha'_+) \in \mathcal{V}'$, the pair 
$$\left(\widetilde{\alpha}'_-, \widetilde{\alpha}'_+\right) \coloneqq \left( \frac{1}{\sqrt{f'}} \alpha'_-, \frac{1}{\sqrt{f'}} \alpha'_+ \right)$$
is in $\mathcal{V}$, where 
$$ \alpha'_\pm \wedge d\alpha'_\pm = \pm f' \, \mathrm{dvol}.$$
Note that $\left(\widetilde{\alpha}'_-, \widetilde{\alpha}'_+\right)$ is a balanced AL pair satisfying $\widetilde{\alpha}'_+ \wedge d \widetilde{\alpha}'_+ = \mathrm{dvol}$. Since $\underline{\ker}$ is open by the previous lemma, $\mathcal{U}' \coloneqq \underline{\ker}(\mathcal{V}') \subset \mathcal{BC}$ is an open neighborhood of $(\xi_-, \xi_+)$. Let $\widetilde{\mathcal{V}}' \subset \mathcal{V}$ be the subspace of elements of the form $\left(\widetilde{\alpha}'_-, \widetilde{\alpha}'_+\right)$ for $(\alpha'_-, \alpha'_+) \in \mathcal{V}'$. One easily checks that $\underline{\ker} : \widetilde{\mathcal{V}}' \rightarrow \mathcal{U}'$ is injective and open. It is surjective by definition, hence it is a homeomorphism. By the previous paragraph, there is an open neighborhood of $(\alpha_-, \alpha_+)$ in $\underline{\ker}^{-1}\big\{(\xi_-, \xi_+)\big\} \cap \mathcal{AL}^b$ homeomorphic to 
$$\Sigma_\epsilon \coloneqq \left\{ \sigma : M \rightarrow \R  :  \vert X \cdot \sigma \vert < \epsilon \right\}$$
for some small $\epsilon > 0$. Therefore, after possibly shrinking $\epsilon$, the map
$$\begin{array}{ccc}
\widetilde{\mathcal{V}}' \times \Sigma_\epsilon & \longrightarrow & \mathcal{AL}^b\\
\big(\big(\widetilde{\alpha}'_-, \widetilde{\alpha}'_+\big), \sigma \big) & \longmapsto & \left(e^\sigma \widetilde{\alpha}'_-,  e^\sigma \widetilde{\alpha}'_+ \right)
\end{array}$$
induces a local trivialization of $\mathcal{AL}^b \rightarrow \underline{\ker}\left(\mathcal{AL} \right)$ around $(\alpha_-, \alpha_+)$.

This proves that $\underline{\ker}$ restricted to $\mathcal{AL}^b$ is a topological submersion with contractible fibers, and the same holds for $\underline{\ker}$ on $\mathcal{AL}$ via the homeomorphism $\vartheta$.
\end{proof}

\begin{rem}
Since $\underline{\ker}$ is open, its image has the homotopy type of a CW complex.
\end{rem}

There is an inclusion $\underline{\ker}\big(\mathcal{AL}\big) \subset \mathbb{P}\mathcal{I}^{-1} \big( \mathcal{AF} \big)$ which is strict according to Theorem~\ref{thm:supporting}.\footnote{Indeed, any volume preserving Anosov flow can be perturbed near a closed orbit in such a way that the new Poincar\'e return map for this orbit has determinant different than $1$, so the flow is not volume preserving anymore.} In more concrete terms, there exist bi-contact structures supporting Anosov flows which cannot be represented as the kernel of a AL pairs. Nevertheless, we have:

\begin{thm} \label{thm:hominc'}
The inclusion $\underline{\ker}\big(\mathcal{AL}\big) \subset \mathbb{P}\mathcal{I}^{-1} \big( \mathcal{AF} \big)$ is a homotopy equivalence.
\end{thm}

\begin{proof} It immediately follows from Theorem~\ref{thm:homeq}, Theorem~\ref{thm:homeqproj}, Theorem~\ref{thm:hominc} and the commutative diagram
$$ \begin{tikzcd}[row sep = large]
\mathcal{AL} \arrow[r, "\sim"] \arrow[drr, "\sim", "\mathcal{I}"', bend right=21]& \underline{\ker} \big(\mathcal{AL}\big) \arrow[r, symbol=\subset] &[-20pt] \mathbb{P}\mathcal{I}^{-1} \big(\mathcal{AF} \big) \arrow[r, hook] \arrow[d, "\rotatebox{90}{$\sim$}","\mathbb{P}\mathcal{I}"'] \arrow[dr, phantom, "\lrcorner", very near start] &[-5pt] \mathcal{BC} \arrow[d, "\rotatebox{90}{$\sim$}","\mathbb{P}\mathcal{I}"']\\
					&			& \mathcal{AF} \arrow[r, hook] & \mathbb{P}\mathcal{AF}
\end{tikzcd}$$
Note that the corestriction of $\mathbb{P}\mathcal{I}$ over $\mathcal{AF}$ is also an acyclic Serre fibration, and all the spaces in this diagram have the homotopy type of CW complexes.
\end{proof}

	\section{Linear Liouville pairs}
	
As explained in the Introduction, there exist other possible definitions for Liouville pairs. The following one is used by some authors (e.g., \cite{M95, H22a}).

\begin{defn} \label{def:linanoliouv}
A pair of contact forms $(\alpha_-, \alpha_+)$ on $M$ is a \textbf{linear Liouville pair} if the $1$-form $$(1-t) \alpha_- + (1+t) \alpha_+$$ on $[-1,1]_t \times M$ is a positively oriented Liouville form. 

The pair $(\alpha_-, \alpha_+)$ is a \textbf{linear Anosov Liouville pair} ($\ell$AL pair for short) if both $(\alpha_-, \alpha_+)$ and $(-\alpha_-, \alpha_+)$ are linear Liouville pairs.
\end{defn}

Note that for this definition, $[-1,1] \times M$ is a Liouville \emph{domain} instead of a Liouville \emph{manifold}. In this section, we study some similarities and differences between Liouville pairs and linear Liouville pairs. In particular, we show that those are two different notions (Lemma~\ref{lem:alvslal}). Moreover, a pair a contact forms which is both a Liouville pair and a linear Liouville pair defines Liouville structures in two different ways, and we show that they are homotopic (Proposition~\ref{prop:sameliouv}). We believe that this result is relevant since all of the natural constructions of Liouville pairs we are aware of satisfy both definitions. The linear formulation might be more convenient in some situations. The results in this section are independent from the main results of this article.

        \subsection{Elementary properties}

The results in Section~\ref{section:elem} can be adapted to $\ell$AL pairs. First of all, $\ell$AL pairs can be characterized in the following way (see Lemma~\ref{lem:charanosov}):

\begin{lem} \label{lem:charanosovlin}
Let $(\alpha_-, \alpha_+)$ be a pair of contact forms on $M$. We write
\begin{align*}
\alpha_+\wedge d \alpha_+ &= f_+ \, \mathrm{dvol}, \\
\alpha_- \wedge d \alpha_- &= - f_- \, \mathrm{dvol}, \\
\alpha_- \wedge d\alpha_+ &= g_+ \ \mathrm{dvol}, \\
\alpha_+\wedge d\alpha_- &= g_- \ \mathrm{dvol}, 
\end{align*}
where $\mathrm{dvol}$ is any volume form on $M$ and $f_\pm, g_\pm : M \rightarrow \R$ are smooth functions.
Then $(\alpha_-, \alpha_+)$ is a $\ell$AL pair if and only if
\begin{align} \label{ineqflin}
|g_-| < f_- \qquad and \qquad |g_+| < f_+.
\end{align}
\end{lem}

\begin{proof}
The pair $(\alpha_-, \alpha_+)$ is a linear Liouville pair if and only if for every $t \in [-1, 1]$, $$(\alpha_+ - \alpha_-) \wedge (d\alpha_+ + d\alpha_- + t(d\alpha_+ - d\alpha_-))>0,$$ 
which is equivalent to $$f_+ +f_- + g_- - g_+  + t(f_+ -f_- - g_- - g_+)>0.$$
This inequality is satisfied for every $t \in [-1,1]$ if and only if it is satisfied for $t=-1$ and $t=1$, which is equivalent to $g_+ < f_+$ and $-g_- < f_-$.
\end{proof}

Similarly to Proposition~\ref{prop:anoliouv}, we also have:

\begin{prop} \label{prop:linanoliouv}
Let $(\alpha_-, \alpha_+)$ be a $\ell$AL pair. Then it defines a bi-contact structure $(\xi_-,\xi_+) = (\ker \alpha_-, \ker \alpha_+)$. Moreover, if $X \in \xi_- \cap \xi_+$ is a nowhere vanishing vector field, then $(X, R_-, R_+)$ is a basis at every point of $M$.
\end{prop}

\begin{proof}
To see that $\xi_-$ and $\xi_+$ are everywhere transverse, we argue by contradiction and assume that they coincide at a point $x\in M$. With the same notations as in the proof of~\ref{prop:anoliouv}, we readily get
\begin{align*}
f_+ &=d\alpha_+(X,Y),\\
f_- &=-\alpha_-(R_+) d\alpha_-(X,Y), \\
g_+ &=\alpha_-(R_+) d\alpha_+(X,Y), \\
g_- &=d\alpha_-(X,Y),
\end{align*}
hence $$\vert g_- g_+\vert =\vert \alpha_-(R_+) d\alpha_-(X,Y) d\alpha_+(X,Y) \vert =  f_- f_+,$$
contradicting~\eqref{ineqflin}.

Now, assuming that $\mathrm{dvol}(X,R_-, R_+) = 0$ at a point $x \in M$, the computations in the proof of Proposition~\ref{prop:anoliouv} show
\begin{align*}
f_- &= -\alpha_-(R_+) g_-, \\
f_+ &= \frac{1}{\alpha_-(R_+)} g_+,
\end{align*}
hence $$|g_- g_+| = f_- f_+,$$
contradicting~\eqref{ineqflin} once again.
\end{proof}

\begin{rem}
A main difference between Anosov Liouville pairs as in Definition~\ref{def:anoliouv} and linear Anosov Liouville pairs as in Definition~\ref{def:linanoliouv} is that there does not seem to be a natural action of $\mathcal{C}^\infty(M, \R)$ on $\ell$AL pairs. Moreover, we do not know if there is a natural modification making a $\ell$AL pair balanced.
\end{rem}

$\ell$AL pairs can also be characterized by their Reeb vector fields (see Proposition~\ref{prop:reebvf}):

\begin{prop} \label{prop:reebvflin}
Let $(\alpha_-, \alpha_+)$ be a pair of contact forms on $M$, negative and positive, respectively. Then it is a $\ell$AL pair if and only if
\begin{align} \label{ineqreeblin}
\vert \alpha_-(R_+)\vert < 1 \qquad and \qquad \vert \alpha_+(R_-)\vert <1.
\end{align}
\end{prop}

\begin{proof}
If $(\alpha_-, \alpha_+)$ is a $\ell$AL pair, then Proposition~\ref{prop:linanoliouv} and the computations in the proof of Proposition~\ref{prop:anoliouv} imply
\begin{align*}
g_+ &= \alpha_-(R_+)\, f_+, \\
g_- &= - \alpha_+(R_-)\, f_-,
\end{align*}
and~\eqref{ineqreeblin} follows from~\eqref{prop:linanoliouv}.

Now, assuming that $(\alpha_-, \alpha_+)$ satisfies~\eqref{ineqreeblin}, it is enough to prove that the conclusions of Proposition~\ref{prop:linanoliouv} are satisfied. This follows exactly from the proof of Proposition~\ref{prop:reebvf}.
\end{proof}

Combining Proposition~\ref{prop:reebvf} and Proposition~\ref{prop:reebvflin}, we obtain that any \emph{balanced} $\ell$AL pair is an AL pair. The converse is \emph{not} true by the following lemma. It also implies that the two definitions of Liouville pairs (Definition~\ref{def:anoliouv} and Definition~\ref{def:linanoliouv}) are \emph{different}:

\begin{lem} \label{lem:alvslal}
Every smooth volume preserving Anosov flow on $M$ admits a supporting balanced AL pair which is \underline{not} a $\ell$AL pair, and whose underlying bi-contact structure is \underline{not} defined by a $\ell$AL pair.
\end{lem}

\begin{proof} Let $(\alpha_s, \alpha_u)$ be a defining pair for a volume preserving Anosov flow $\Phi = \{\phi^t\}$. For $A \geq 1$, we define
\begin{align*}
\alpha_- &\coloneqq e^{-A} \alpha_u + e^A \alpha_s,\\
\alpha_+ &\coloneqq e^A \alpha_u - e^{-A} \alpha_s.
\end{align*}
If $\mathrm{dvol}$ is such that $\iota_X \mathrm{dvol} = \alpha_s \wedge \alpha_u$, where $X$ is the vector field generating the flow, then one easily computes
\begin{align*}
f_\pm &= r_u - r_s = 2 r_u, \\ 
g_\pm &= - 2 \sinh(2A) r_u, 
\end{align*}
so $(\alpha_-, \alpha_+)$ is a $\mathcal{C}^1$ closed balanced AL pair, but it is not a $\ell$AL pair since $\vert g_\pm \vert > f_\pm$. This remains true for a suitable smoothing of $(\alpha_-, \alpha_+)$.

Let us assume for simplicity that the pair $(\alpha_-, \alpha_+)$ as above is smooth. We show that there are no functions $h_\pm : M \rightarrow \R_{>0}$ such that $(h_- \alpha_-, h_+ \alpha_+)$ is a $\ell$AL pair. Indeed, let us assume by contradiction that such functions exist. By Lemma~\ref{lem:charanosovlin} they would satisfy the following inequalities:
\begin{align*}
\vert \cosh(2A) h_+ X\cdot h_-  - \sinh(2A) r_u h_- h_+  \vert &< r_u h_-^2, \\
\vert \cosh(2A)  h_- X\cdot h_+ + \sinh(2A) r_u h_- h_+  \vert &< r_u h_+^2.
\end{align*}
Writing $\rho_\pm \coloneqq 1/h_\pm$, these are equivalent to 
\begin{align*}
\vert X\cdot \rho_-  +  \tanh(2A) r_u  \rho_-\vert &< \frac{r_u}{\cosh(2A)} \rho_+, \\
\vert  X\cdot \rho_+ - \tanh(2A) r_u \rho_+\vert &< \frac{r_u}{\cosh(2A)} \rho_-.
\end{align*}
Fixing a point $x \in M$, we define $y_\pm : \R \rightarrow \R_{>0}$ by $$y_\pm(t) \coloneqq \rho_\pm \circ \phi^t(x).$$ There exists $C  >0$ such that $0 <y_\pm < C$. Moreover, these functions satisfy
\begin{align*}
\left\vert \dot{y}_-  +  a y_- \right\vert &< \epsilon y_+, \\
\left\vert  \dot{y}_+ - a y_+ \right\vert &< \epsilon y_-,
\end{align*}
where 
\begin{align*}
a(t) &\coloneqq \tanh(2A) \, r_u \circ \phi^t (x)> 0, \\
\epsilon(t) &\coloneqq \frac{r_u \circ \phi^t(x)}{\cosh(2A)} > 0.
\end{align*}
Since $A \geq 1$, we have $a > 2 \epsilon$.
It follows that for every $T>0$, 
\begin{align*}
\int_0^T ay_+ \, dt &\leq \int_0^T \epsilon y_- \, dt + y_+(T) + y_0(T) \\
&\leq \frac{1}{2} \int_0^T a y_- \, dt + 2C \\
&\leq \frac{1}{2} \int_0^T \epsilon y_+ \, dt + 3C \\
& \leq \frac{1}{4} \int_0^T a y_+ \, dt + 3C,
\end{align*}
hence
$$ \int_0^T a y_+ \, dt \leq 4C.$$
However, $a y_+$ is bounded from below by some positive constant, which contradicts the previous inequality for $T$ large enough.

If $(\alpha_-, \alpha_+)$ is only $\mathcal{C}^1$, this strategy still applies to a suitable smoothing of $(\alpha_-, \alpha_+)$ which is sufficiently $\mathcal{C}^1$-close to $(\alpha_-, \alpha_+)$.
\end{proof}

\begin{rem}
We also expect that there exist (unbalanced) $\ell$AL pairs which are not AL pairs, but the construction seems more delicate.
\end{rem}

        \subsection{Induced Liouville structures}

The ``standard construction'' of Section~\ref{sec:std} yields a pair of contact forms which is both an AL pair and a $\ell$AL pair (after smoothing). If $(\alpha_-, \alpha_+)$ is a pair of contact forms which is both a Liouville pair and a linear Liouville pair, we can consider two Liouville structures on $\R_s \times M$:
\begin{enumerate}[label=(\arabic*)]
    \item The \emph{completion} $\widehat{\lambda}_\mathrm{lin}$ of the Liouville domain $[-1,1]_t \times M$ with the ``linear'' Liouville form $$\lambda_\mathrm{lin} \coloneqq (1-t) \alpha_- + (1+t) \alpha_+,$$
    \item The Liouville structure induced by the ``exponential'' Liouville form $$\lambda_\mathrm{exp} \coloneqq e^{-s} \alpha_- + e^s \alpha_+.$$
\end{enumerate}

Here, the completion of a Liouville domain $V$ is obtained by attaching to $\partial V$ the symplectization $[0, \infty) \times \partial V$ of the contact structure at the boundary. This procedure yields an open manifold $\widehat{V}$ with controlled geometry at infinity. See~\cite[Section 11.1]{CE12} for a precise definition. The next proposition shows in particular that (1) and (2) above produce equivalent Liouville structures on $\R \times M$.

\begin{prop} \label{prop:sameliouv}
Let $(\alpha_-, \alpha_+)$ be a pair of contact forms which is both a Liouville pair and a linear Liouville pair. The Liouville structures $\widehat{\lambda}_\mathrm{lin}$ and $\lambda_\mathrm{exp}$ on $\R \times M$ are Liouville homotopic.
\end{prop}

\begin{proof} We choose an arbitrary volume form $\mathrm{dvol}$ and we define $f_\pm$, $g_\pm$ as in Lemma~\ref{lem:charanosovlin}. We also choose $A > 0$ such that $f_\pm < A$ and $\vert g_\pm \vert < A$. We proceed in 3 steps.

\textit{Step 1 : extending $\lambda_\mathrm{lin}$.} Let $\epsilon > 0$. We choose a smooth function $\phi = \phi_\epsilon : \R \rightarrow \R_{\geq 0}$ satisfying
\begin{itemize}
\item $\phi(s) = 0$ for $s \leq -1-\epsilon$,
\item $\phi(s)>0$ for $s > -1-\epsilon$,
\item $\phi(s) = 1+s$ for $s \geq -1+\epsilon$,
\item $\phi$ is non-decreasing and convex.
\end{itemize} 

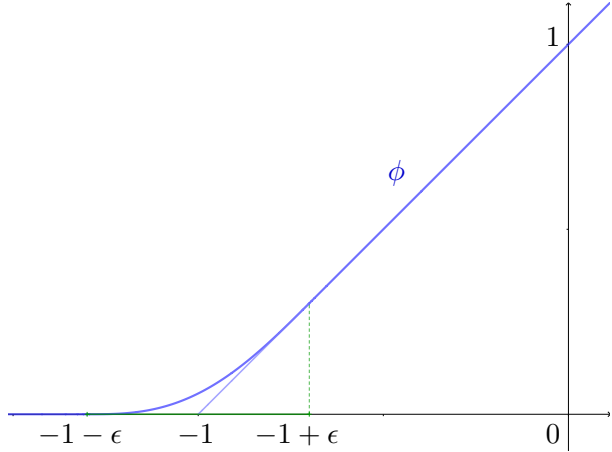
\begin{figure}[t]
    \begin{center}
        \begin{picture}(80, 60)(0,0)
        \put(0,0){\includegraphics[width=80mm]{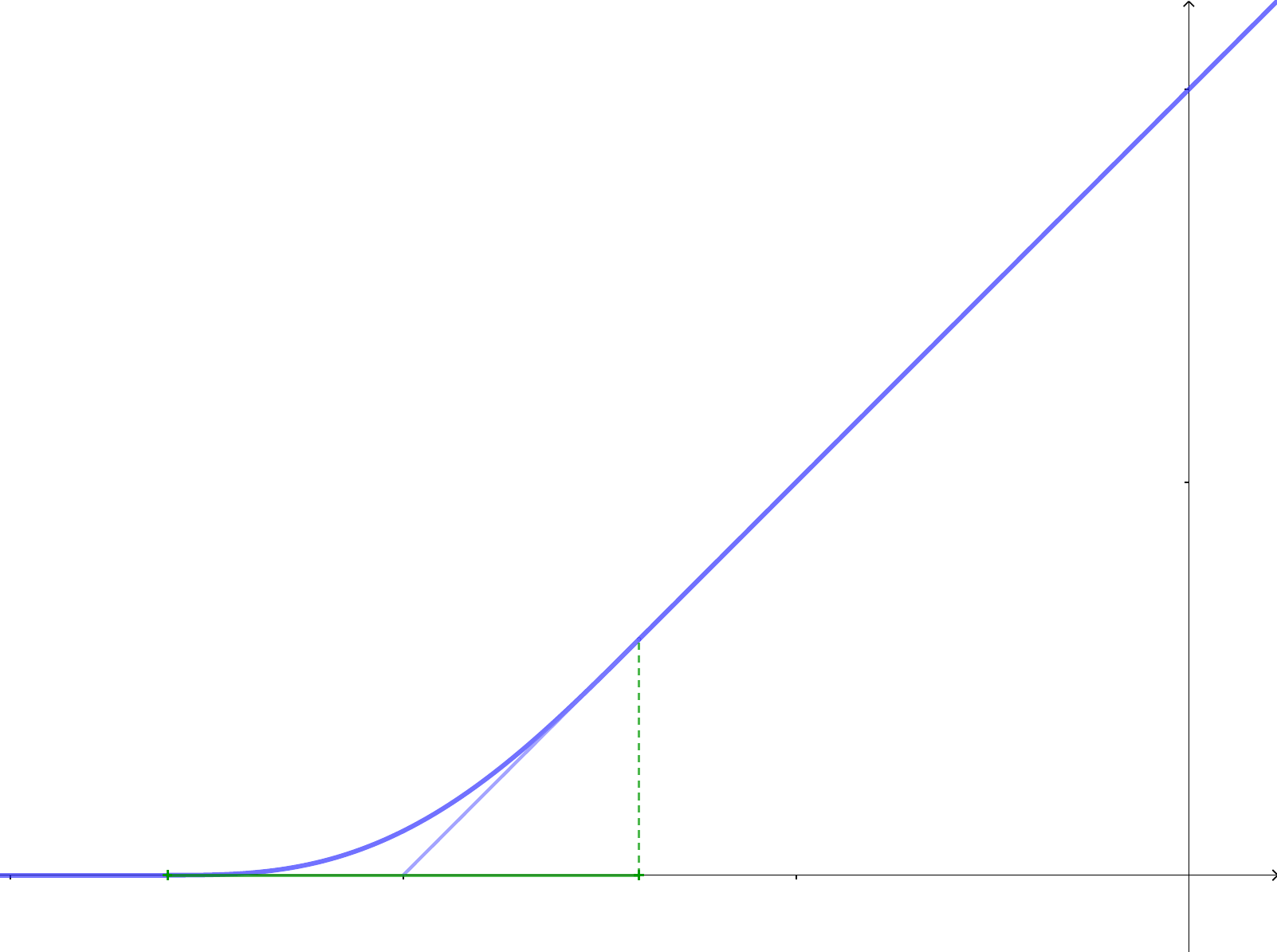}}
        \put(4,1){$-1-\epsilon$}
        \put(22.5,1){$-1$}
        \put(32.8,1){$-1+\epsilon$}
        \put(71.5,1){$0$}
        \put(71.5, 53.8){$1$}
        \color{blue1}
        \put(50.5,36){$\phi$}
        \end{picture}
        \captionsetup{width=100mm}
        \caption{The function $\phi$.}
    \end{center}
    \label{fig:phi}
\end{figure}

We claim that for $\epsilon$ sufficiently small, the $1$-form 
$$\lambda_0 \coloneqq \phi(-s) \alpha_- + \phi(s) \alpha_+$$
on $\R_s \times M$ is Liouville. On $[-1+\epsilon, 1-\epsilon] \times M$, it coincides with $\lambda_\mathrm{lin}$ which is Liouville. On $(-\infty, -1-\epsilon] \times M$, it coincides with $(1-s) \alpha_-$ which is Liouville since $\alpha_-$ is a negative contact form. On $[1+ \epsilon, \infty) \times M$, it coincides with $(1+s) \alpha_+$ which is Liouville since $\alpha_+$ is positive contact form. If $s \in [1-\epsilon, 1+\epsilon]$, we compute:
\begin{align*}
d\lambda_0 \wedge d\lambda_0 &= \left\{ (1+s) \big(f_+ -\phi'(-s) g_+\big) + \phi(-s) \big(\phi'(-s) f_- + g_-\big) \right\} ds \wedge \mathrm{dvol} \\
&= F \, ds \wedge \mathrm{dvol}.
\end{align*}
Note that 
$$ 0 \leq \phi(-s) \leq \epsilon, \qquad 0 \leq \phi'(-s) \leq 1,$$
hence
\begin{align*}
    F &\geq (2-\epsilon) \min\{f_+, f_+ - g_+\} + 2 \epsilon A,
\end{align*}
where $f_+ > 0$ and $f_+ - g_+ > 0$ by Lemma~\ref{lem:charanosovlin}, thus $F>0$ for $\epsilon$ small enough. The case $s \in [-1-\epsilon, -1 + \epsilon]$ is similar.

\textit{Step 2: $\lambda_0$ is Liouville homotopic to $\widehat{\lambda}_\mathrm{lin}$.} We will use the following elementary fact: if two Liouville structures on a manifold with boundary $V$ are Liouville homotopic, then their completions on $\widehat{V}$ are Liouville homotopic; see~\cite[Lemma 11.6]{CE12}. By definition, $([-1-\epsilon, 1+\epsilon] \times M, \lambda_0)$ is a Liouville domain whose completion is exactly $(\R \times M, \lambda_0)$. Moreover, if $\epsilon$ is small enough, then $([-1-\epsilon, 1+\epsilon] \times M, \lambda_0)$ and $([-1+\epsilon, 1-\epsilon] \times M, \lambda_0) = ([-1+\epsilon, 1-\epsilon] \times M, \lambda_\mathrm{lin})$ are Liouville domains which are Liouville homotopic (after identifying $[-1-\epsilon, 1+\epsilon]$ and $[-1+\epsilon, 1-\epsilon]$), and $([-1+\epsilon, 1-\epsilon] \times M, \lambda_\mathrm{lin})$ and $([-1,1] \times M, \lambda_\mathrm{lin})$ are also Liouville homotopic (after identifying $[-1+\epsilon, 1-\epsilon]$ and $[-1,1]$). This shows that $([-1-\epsilon, 1+\epsilon] \times M, \lambda_0)$ and $([-1,1] \times M, \lambda_\mathrm{lin})$ are Liouville homotopic (after identifying $[-1-\epsilon, 1+\epsilon]$ with $[-1,1]$), and so are their completions.

\textit{Step 3: $\lambda_0$ is Liouville homotopic to $\lambda_\mathrm{exp}$.}
For $\tau \in [0,1]$, we set $$\psi_\tau(s) \coloneqq \tau e^s + (1-\tau) \phi(s),$$ and 
$$\lambda_\tau \coloneqq \psi_\tau(-s) \alpha_- + \psi_\tau(s) \alpha_+.$$
The family $\{\lambda_\tau\}_{\tau \in [0,1]}$ interpolates between $\lambda_0$ and $\lambda_1=  e^{-s} \alpha_- + e^s \alpha_+$. It is enough to show that for every $\tau \in (0,1)$, $d \lambda_\tau \wedge d\lambda_\tau$ is a positive volume form on $\R \times M$. By symmetry, it is enough to show it for $s \geq 0$. The computation of $d\lambda_\tau \wedge d\lambda_\tau$ reveals that the latter is equivalent to
\begin{align} \label{ineq:atbt}
f_+ + a_\tau(s) g_- - b_\tau(s) g_+ + a_\tau(s)b_\tau(s) f_- > 0,
\end{align}
where
\begin{align*}
a_\tau(s) \coloneqq \frac{\psi_\tau(-s)}{\psi_\tau(s)}, \qquad b_\tau(s) = \frac{\psi'_\tau(-s)}{\psi'_\tau(s)}.
\end{align*}
It is easy to check that for $\tau \in (0,1)$ and $s \geq 0$, $0 \leq a_\tau(s) \leq 1$ and $0 \leq b_\tau(s) \leq 1$.
Since $(\alpha_-, \alpha_+)$ is both an exponential and a linear Liouville pair, we have that for every $a \in [0,1]$,
\begin{align*}
f_+ + a g_- - ag_+ + a^2 f_- &> 0,\\
f_+ + a g_- - g_+ + a f_- &> 0,
\end{align*}
so for every $a \in [0,1]$ and $b \in [a,1]$, we have
\begin{align}
 f_+ + a g_- - bg_+ + ab f_- > 0. \label{ineq:ab}
 \end{align}
By compactness, there exists $\delta > 0$, only depending on $(\alpha_-, \alpha_+)$, such that for every $a \in [0,1]$ and $b \in [a-\delta, 1]$, the inequality \eqref{ineq:ab} is satisfied. We claim that for every $\tau \in (0,1)$ and $s\geq 0$, we have
\begin{align} \label{ineq:weird}
    b_\tau(s) - a_\tau(s) \geq  - \epsilon.
\end{align}
Indeed, fixing $\tau \in (0,1)$, we consider two cases.

\textit{Case 1.} If $s \in [0, 1- \epsilon)\cup [1+\epsilon, \infty)$,
        \begin{align*}
            \psi_\tau(s) &\geq \psi'_\tau(s),\\ 
            \psi'_\tau(-s) &\geq \psi_\tau(-s),
        \end{align*}
        and~\eqref{ineq:weird} follows trivially since the left-hand side is non-negative.

\textit{Case 2.} If $s \in [1-\epsilon, 1+\epsilon)$, 
            \begin{align*}
                \psi_\tau(s) &= \tau e^s + (1-\tau) (1+s) \geq 1,\\
                \psi_\tau'(s) &= \tau e^s + (1-\tau) \geq 1,
            \end{align*}
            and we compute
            \begin{align*}
                \psi'_\tau(-s) \psi_\tau(s) &= \tau^2 + \tau(1-\tau) \left\{(1+s) e^{-s} + e^s \phi'(-s) \right\} + (1-\tau)^2 (1+s) \phi'(-s),\\
                \psi'_\tau(s) \psi_\tau(-s) &= \tau^2 + \tau(1-\tau) \left\{e^s \phi(-s) + e^{-s}\right\} + (1-\tau)^2 \phi(-s),
            \end{align*}
            hence
            \begin{align*}
                \psi'_\tau(-s) \psi_\tau(s) - \psi'_\tau(s) \psi_\tau(-s) &= \begin{multlined}[t]  (1-\tau)\Big\{\tau \Big( \, \overset{(1)}{\overbrace{e^s \big(\phi'(-s) - \phi(-s)\big) + s e^{-s}}} \, \Big)  \\ + (1-\tau) \Big( \, \underset{(2)}{\underbrace{(1+s)\phi'(-s) - \phi(-s)}} \, \Big) \Big\}.
                \end{multlined}
            \end{align*}
            Since $\phi'(-s) \geq 0$ and $0 \leq \phi(-s) \leq \epsilon$, 
            \begin{align*}
            \mathrm{(1)} &\geq - e^{1+\epsilon} \epsilon + (1-\epsilon) e^{-(1+\epsilon)}, \\
            \mathrm{(2)} & \geq - \epsilon.
            \end{align*}
            For $\epsilon$ small enough, say $\epsilon \leq 1/100$, $\mathrm{(1)} \geq 0$, and~\eqref{ineq:weird} follows.

This shows that for $\epsilon$ small enough, only depending on $(\alpha_-, \alpha_+)$, the inequality~\eqref{ineq:atbt} is satisfied for every $\tau \in (0,1)$ and $s\geq 0$. The case $s \leq 0$ can be treated similarly.
\end{proof}
	
\begin{rem}
As mentioned in the introduction, we do not know if Theorem~\ref{thm:contract} and Theorem~\ref{thm:homeq} are also true for linear Anosov Liouville pairs. The proof of Theorem~\ref{thm:homeq} would immediately adapt to the linear case, provided that the space of $\ell$AL pairs supporting a given flow is (weakly) contractible. Our attempts at proving this fact for $\ell$AL pairs were fruitless because of the complexity and the lack of symmetry of the equations we obtained.
\end{rem}

	\appendix
	\section{Smoothing lemmas} \label{appendixA}

This appendix concerns useful smoothing lemmas which are required to extend the results of this paper to Anosov flows generated by $\mathcal{C}^1$ vector fields, as their weak-stable and weak-unstable bundles are not necessarily $\mathcal{C}^1$. The approach can also be used to bypass Hozoori's delicate approximation techniques in~\cite[Section 4]{H22a}. We state the results in greater generality than needed. $M$ now denotes a closed $n$-dimensional manifold ($n \geq 1$) and $X$ denotes a non-singular vector field on $M$ of class $\mathcal{C}^k$, $1 \leq k \leq \infty$ (without any Anosovity condition). We fix an arbitrary auxiliary metric on $M$.

The first smoothing lemma follows from~\cite[Lemma 4.3]{H22a} and the regular approximation of differentiable functions by smooth ones.

\begin{lem} \label{lem:smooth1}
Let $f : M \rightarrow \R$ be a continuous function which is continuously differentiable along $X$. Then for every $\epsilon > 0$, there exists a smooth function $f^\epsilon : M \rightarrow \R$ satisfying
	\begin{align*}
	\vert f^\epsilon - f \vert_{\mathcal{C}^0} \leq \epsilon \qquad and \qquad \vert X \cdot f^\epsilon - X \cdot f \vert_{\mathcal{C}^0} \leq \epsilon.
	\end{align*}
\end{lem}

In other words, with the notations of Definition~\ref{def:reg}, $\mathcal{C}^\infty$ is dense in $\mathcal{C}^0_X$. The same holds with $\mathcal{C}^\ell_X$ in place of $\mathcal{C}^0_X$, for $0 \leq \ell \leq k-1$. We will need a similar result for $1$-forms on $M$.

\begin{lem} \label{lem:smooth2}
The space of $\mathcal{C}^k$ $1$-forms on $M$ vanishing along $X$ is dense in $\Omega_{X,\ell}^1$ for $0 \leq \ell \leq k-1$.
\end{lem}

\begin{proof}
This is a straightforward adaptation of the proof of~\cite[Lemma 4.3]{H22a}.

By compactness of $M$, we can find a positive real number $\tau > 0$ and a finite collection $\{ (U_i, V_i, \phi_i)\}_{1 \leq i \leq N}$ where
\begin{itemize}
	\item $V_i \subset U_i \subset M$ are open subsets of $M$,
	\item $\{V_i\}_{1 \leq i \leq N}$ is a covering of $M$,
	\item $\phi_i : U_i \rightarrow  (-2\tau, 2\tau)_t \times D$ is a $\mathcal{C}^k$ diffeomorphism such that $\phi_i(V_i) = (-\tau, \tau) \times D$ and $d\phi_i(X) = \partial_t$. Here, $D$ denotes the open unit disk in $\R^{n-1}$.
\end{itemize}
Such a collection can be obtained by taking a finite collection of sufficiently small flow-boxes for $X$ covering $M$.

Let $\{\psi_i\}_{1 \leq i \leq N}$ be a partition of unity subordinate to the open covering $\{V_i\}_{1 \leq i \leq N}$. For every $i$, the support of $\psi_i$ is contained in $V_i$ so we can find $0 < r < 1$ such that the support of $\psi_i \circ \phi_i^{-1}$ is contained in $(-\tau, \tau) \times D_r$, where $D_r$ denotes the open disk of radius $r$.

Let $h : \R \rightarrow \R_{\geq 0}$ be a smooth bump function satisfying
\begin{itemize}
\item For $|t| \leq \tau$, $h(t) = 1$,
\item For $|t| \geq 2\tau $, $h(t) = 0$,
\item $h$ is non-decreasing on $(-\infty, 0)$ and non-increasing on $(0, \infty)$.
\end{itemize}

\begin{figure}[t]
    \begin{center}
        \begin{picture}(80, 54)(0,0)
        \put(0,3){\includegraphics[width=80mm]{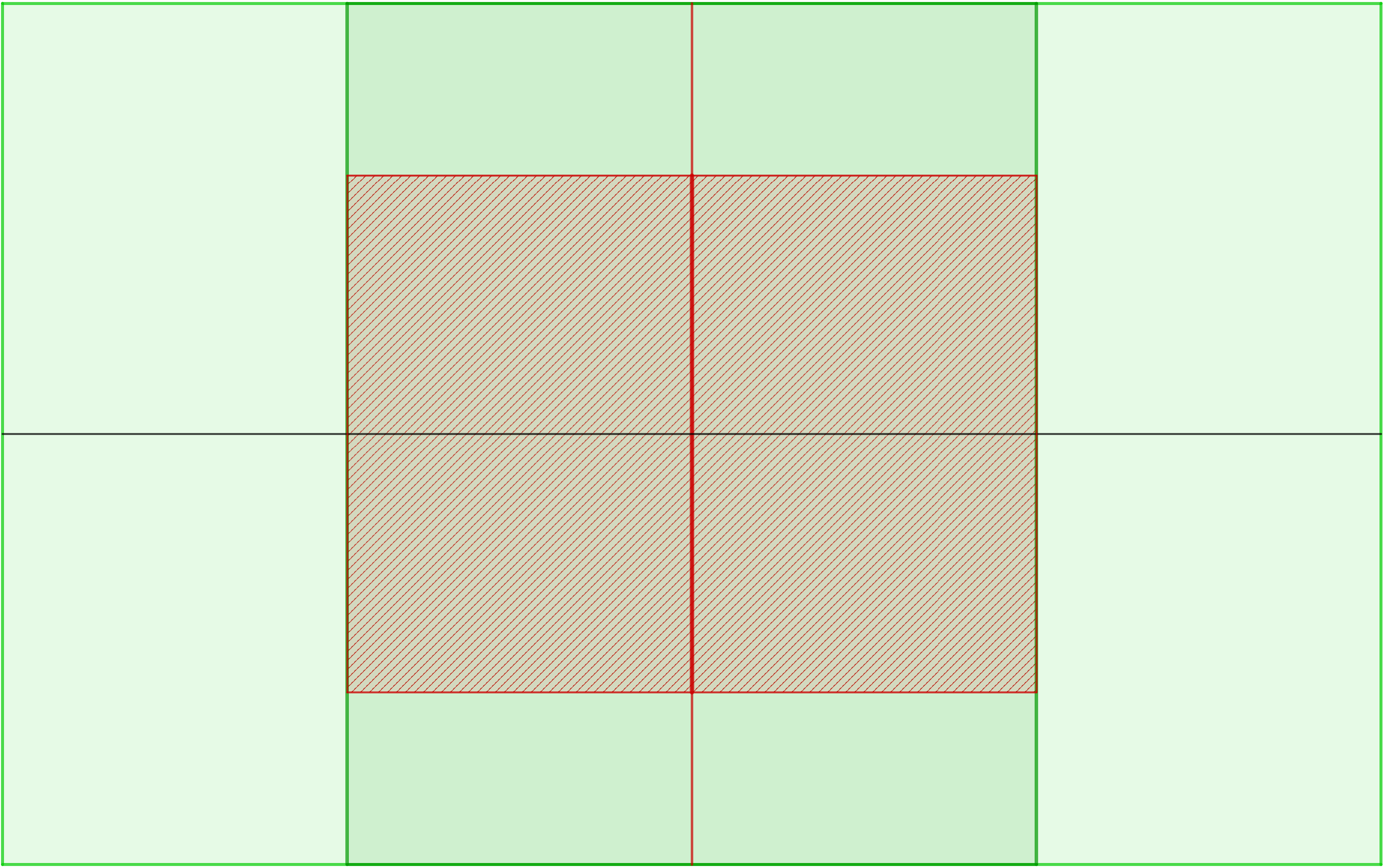}}
        \put(-3,0){$-2\tau$}
        \put(17,0){$-\tau$}
        \put(39,0){$0$}
        \put(59,0){$\tau$}
        \put(77,0){$2\tau$}
        \color{red1}
        \put(41,35){$D_r$}
        \put(41,47){$D$}
        \end{picture}
        \captionsetup{width=100mm}
        \caption{The nested open sets in the argument. The support of $\alpha'_i$ is contained in the hashed region.}
    \end{center}
    \label{fig:box}
\end{figure}
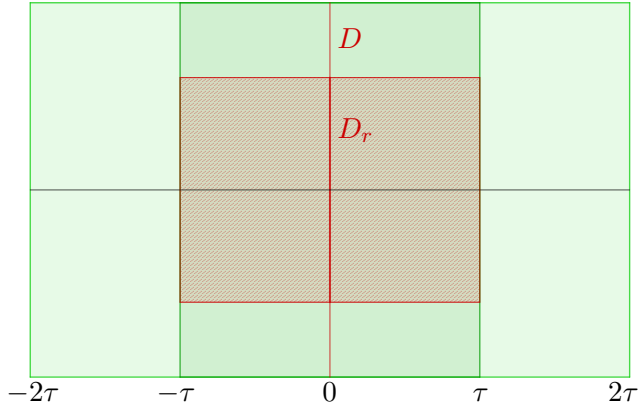

Let $\alpha \in \Omega^1_{X, \ell}$. By definition, $\alpha = \sum_{i=1}^N \psi_i \alpha$. For every $i$, we write $\alpha_i \coloneqq \psi_i \alpha$ and $\alpha'_i \coloneqq (\phi_i)_* \alpha_i$. We have reduced the problem to a single flow-box $(-2\tau, 2\tau) \times D$.

Let $\epsilon > 0$. In what follows, the symbol ``$\lesssim$'' means ``less or equal than up to a constant factor that does not depend on $\epsilon$''. For a fixed $i$, let $\beta^0_i$ be a smooth $1$-form on $D$ with support contained in  $D_r$ satisfying
$$\left \vert \beta^0_i - {\alpha_i'}_{\vert \{0\} \times D} \right \vert_{\mathcal{C}^\ell} \leq \epsilon.$$ 
By definition, $\mathcal{L}_{\partial t} \alpha_i'$ is $\mathcal{C}^\ell$ and vanishes along $\partial_t$. Therefore, we can find a smooth $1$-form $\eta_i$ on $(-2 \tau, 2 \tau)\times D$ with support contained in $(-\tau, \tau) \times D_r$ satisfying
\begin{align*}
\eta_i(\partial_t) &= 0, \\
\left  \vert \eta_i - \mathcal{L}_{\partial t} \alpha_i' \right \vert_{\mathcal{C}^\ell} &\leq \epsilon.
\end{align*}
We can extend $\beta^0_i$ to a smooth $1$-form $\beta_i$ on $(-2\tau, 2\tau) \times D$ by setting
\begin{align*}
\beta_i(\partial_t) &\coloneqq 0, \\
\mathcal{L}_{\partial t} \beta_i &\coloneqq \eta_i.
\end{align*} 
Finally, we define $\beta_i' \coloneqq h  \beta_i$. This is a smooth $1$-form with support in $(-2\tau, 2\tau) \times D_r$. Note that at a point $(t,x) \in (-2\tau, 2\tau) \times D$, we have
\begin{align*}
\big(\beta_i - \alpha_i'\big)_{(t,x)} = \left(\beta_i^0 - {\alpha_i'}_{\vert \{0\} \times D}\right)_{x} + \int_0^t \left( \eta_i - \mathcal{L}_{\partial t} \alpha_i' \right)_{(s,x)} ds,
\end{align*}
so
\begin{align*}
\left \vert \beta_i - \alpha_i'  \right \vert_{\mathcal{C}^\ell} &\leq \left \vert \beta_i^0 - {\alpha_i'}_{\vert \{0\} \times D} \right \vert_{\mathcal{C}^\ell} + 2\tau \left \vert \eta_i - \mathcal{L}_{\partial t} \alpha_i' \right \vert_{\mathcal{C}^\ell} \lesssim \epsilon, \\
\left \vert \mathcal{L}_{\partial t} \beta_i -  \mathcal{L}_{\partial t} \alpha_i' \right \vert_{\mathcal{C}^\ell} &=  \left \vert \eta_i -  \mathcal{L}_{\partial t} \alpha_i' \right \vert_{\mathcal{C}^\ell} \leq \epsilon.
\end{align*}
Since the support of $\alpha_i'$ is contained in $(-\tau, \tau) \times D_r$, we readily get
\begin{align*}
\left \vert {\beta_i}_{\vert ((-2\tau, -\tau) \cup (\tau, 2\tau))\times D} \right \vert_{\mathcal{C}^\ell} \lesssim \epsilon.
\end{align*}
Moreover, 
\begin{itemize}
\item On $((-2\tau, -\tau) \cup (\tau, 2\tau))\times D$,
	\begin{align*}
	\beta_i' - \alpha_i' &= h \beta_i,\\
	\mathcal{L}_{\partial t} \beta_i' - \mathcal{L}_{\partial t} \alpha_i' &= (\partial_t h) \beta_i,
	\end{align*}
\item On $(-\tau, \tau) \times D$,
	\begin{align*}
	\beta_i' - \alpha_i' &= \beta_i - \alpha_i',\\
	\mathcal{L}_{\partial t} \beta_i' - \mathcal{L}_{\partial t} \alpha_i' &=  \eta_i - \mathcal{L}_{\partial t} \alpha_i',
	\end{align*}
\end{itemize}
and we obtain
\begin{align*}
\left \vert \beta_i' - \alpha_i'  \right \vert_{\mathcal{C}^\ell} &\lesssim \epsilon, \\
\left \vert \mathcal{L}_{\partial t} \beta_i' -  \mathcal{L}_{\partial t} \alpha_i' \right \vert_{\mathcal{C}^\ell} & \lesssim \epsilon.
\end{align*}
Finally, we define
$$\beta \coloneqq \sum_i \phi_i^* \beta_i',$$
so that $\beta$ is a $\mathcal{C}^k$ $1$-form on $M$ satisfying $\beta(X) = 0$, and
\begin{align*}
\left \vert \beta - \alpha \right \vert_{\mathcal{C}^\ell} &\leq \sum_{i=1}^N \left \vert \phi_i^* \beta_i' - \phi_i^* \alpha_i' \right \vert_{\mathcal{C}^\ell} \lesssim \sum_{i=1}^N \left \vert \beta_i' - \alpha_i' \right \vert_{\mathcal{C}^\ell} \lesssim \epsilon,\\
\left \vert \mathcal{L}_X  \beta - \mathcal{L}_X  \alpha \right \vert_{\mathcal{C}^\ell} & \leq \sum_{i=1}^N \left \vert \phi_i^* (\mathcal{L}_{\partial t}  \beta_i') -  \phi_i^* (\mathcal{L}_{\partial t} \alpha_i') \right \vert_{\mathcal{C}^\ell} \lesssim \sum_{i=1}^N \left \vert \mathcal{L}_{\partial t} \beta_i' -  \mathcal{L}_{\partial t} \alpha_i' \right \vert_{\mathcal{C}^\ell} \lesssim \epsilon.
 \end{align*}
This finishes the proof.
\end{proof}

	\section{Almost volume preserving Anosov flows} \label{appendixB}

In this appendix, we prove a technical result used in the proof of Theorem~\ref{thm:supporting}. Let us recall the setup. $\Phi$ is a smooth Anosov flow on a closed oriented $3$-manifold $M$, generated by a vector field $X$. For an adapted metric $g$, $r_u > 0$ and $r_s<0$ denote the expansion rates in the unstable and stable directions respectively. The divergence of $X$ for this metric is simply $\mathrm{div}_g X = r_u + r_s$. We say that $\Phi$ is \textbf{almost volume preserving} if it satisfies one of the following equivalent conditions (compare with~\eqref{eq:cond}):

\medskip

\begin{itemize}
\item[(C1)] For every $\epsilon >0$, there exists a smooth function $f_\epsilon : M \rightarrow \R$ satisfying \label{condC1}
\begin{align*}
\left\vert \mathrm{div}_g X + X \cdot f_\epsilon \right\vert \leq \epsilon.
\end{align*}
\item[(C2)] For every $\epsilon >0$, there exists a smooth volume form $\mathrm{dvol}_\epsilon$ on $M$ satisfying 
\begin{align*}
\vert \mathrm{div}_\epsilon X \vert = \left\vert \frac{\mathcal{L}_X \mathrm{dvol}_\epsilon}{\mathrm{dvol}_\epsilon } \right\vert \leq \epsilon,
\end{align*}
where $\mathrm{div}_\epsilon X$ denotes the divergence of $X$ with respect to $\mathrm{dvol}_\epsilon$.
\end{itemize}

\begin{prop} If $\Phi$ is a smooth almost volume preserving Anosov flow on $M$, then $\Phi$ is volume preserving.
\end{prop}

\begin{proof} As noted in the proof of Theorem~\ref{thm:supporting}, it is enough to show that $\Phi$ is topologically transitive, i.e., $\Phi$ has a dense orbit. We closely follow the strategy from~\cite{M22} that relies on some key properties of Sina\"{i}--Ruelle--Bowen (SRB) measures for Anosov diffeomorphisms. Here, the time-one map $\phi=\phi^1$ of $\Phi$ is not Anosov since it is neither contracting nor expanding in the direction of the flow. Nevertheless, the main arguments of~\cite{M22} can be adapted to our context with only minor modifications.

A Borel probability measure $\mu$ on $M$ preserved by $\phi$ is a \emph{SRB measure} if it has absolutely continuous conditional measures on unstable manifolds, see~\cite[Definition 1.4.2]{LY85}. Here, the (strong) unstable manifolds of $\Phi$ and $\phi$ coincide. By~\cite[Theorem C]{CY05}, $\phi$ admits a SRB measure. Let $g$ be a Riemannian metric on $M$ adapted to $\Phi$ (so that $r_s < 0 < r_u$) and $\mu$ be any Borel probability measure preserved by $\phi$. For $\mu$-a.e. $x \in M$ and for unit vectors $v_u \in E^u(x)$ and $v_s \in E^u(x)$, the \emph{Lyapunov exponents}
\begin{align*}
\Lambda_u(x) &\coloneqq \lim_{n \rightarrow \pm \infty} \frac{1}{n} \ln \Vert d\phi^n (v_u) \Vert = \lim_{n \rightarrow \pm \infty} \frac{1}{n} \int_0^n r_u \circ \phi^s(x) \, ds, \\
\Lambda_s(x) &\coloneqq \lim_{n \rightarrow \pm \infty} \frac{1}{n} \ln \Vert d\phi^n (v_s) \Vert = \lim_{n \rightarrow \pm \infty} \frac{1}{n}\int_0^n r_s \circ \phi^s(x) \, ds,
\end{align*}
exist, are finite, independent of the metric, and satisfy $\Lambda_s(x) < 0 < \Lambda_u(x)$. The third Lyapunov exponent corresponding to the direction of the flow is $0$.

The hypothesis on $\Phi$ implies that $\Lambda_u + \Lambda_s = 0$ $\mu$-a.e. Indeed, let $\epsilon > 0$ and choose a smooth (hence bounded) function $f_\epsilon : M \rightarrow \R$ such that $| r_u + r_s + X \cdot f_\epsilon| < \epsilon$. Then $\mu$-a.e., we have
\begin{align*}
\vert \Lambda_u + \Lambda_s\vert &\leq \limsup_{n \rightarrow \infty} \frac{1}{n}\int_0^n \vert r_u + r_s + X \cdot f_\epsilon \vert \circ \phi^s \, ds  + \limsup_{n \rightarrow \infty} \frac{1}{n} \left \vert \int_0^n (X \cdot f_\epsilon) \circ \phi^s \, ds \right \vert \\
&\leq \epsilon + \limsup_{n \rightarrow \infty} \frac{1}{n} \left \vert f_\epsilon \circ \phi^n - f_\epsilon \right \vert\\
&= \epsilon.
\end{align*}

By~\cite[Theorem A]{LY85}, $\mu$ is a SRB measure for $\phi$ if and only if it satisfies the \emph{Pesin entropy formula}
\begin{align*}
h_\mu(\phi) = \int_M \Lambda_u \, d\mu.
\end{align*}
Assuming that $\mu$ is a SRB measure for $\phi$, it follows that
\begin{align*}
h_\mu\big(\phi^{-1}\big) = \int_M - \Lambda_s \, d\mu,
\end{align*}
so since $- \Lambda_s$ corresponds to the unstable Lyapunov exponent of $\phi^{-1}$, $\mu$ is also a SRB measure for $\phi^{-1}$. By~\cite[Th\'{e}or\`{e}me (5.5)]{L84}, $\mu$ is absolutely continuous with respect to the Lebesgue measure on $M$ (induced by some Riemannian metric). If $\Omega$ denotes the non-wandering set of $\Phi$, then $\mu(\Omega) = 1$ so the Lebesgue measure of $\Omega$ is positive. Smale's spectral decomposition for $\Omega$ states that $\Omega$ can be partitioned into finitely many \emph{basic sets} $\Omega_1, \dots, \Omega_N$. Those are compact invariant subsets on which $\Phi$ is transitive. It follows that there exists a basic set $\Omega_i$ with positive Lebesgue measure. By~\cite[Corollary 5.7 (b)]{BR75}, $\Omega_i = M$ and $\Phi$ is transitive. This concludes the proof.
\end{proof}

\begin{rem} The proof can be adapted to show that any almost volume preserving Anosov flow of class $\mathcal{C}^2$  on any closed manifold is volume preserving.  In higher dimension, $\Lambda_u$  should be replaced by the sum of the positive Lyapunov exponents (counted with multiplicity) and $\Lambda_s$ by the sum of the negative ones.
\end{rem}

\begin{rem} The main result of~\cite{M22} asserts that an Anosov \emph{diffeomorphism} of class $\mathcal{C}^2$ on a closed manifold $M$ satisfying that at every periodic point, the Poincar\'{e} return map has determinant one is volume preserving. The same result remains true for an Anosov flow whose Poincar\'{e} return map at every closed orbit has determinant one. As in the proof in~\cite{M22}, this condition and Anosov's \emph{shadowing property} imply that $\Lambda_u + \Lambda_s=0$ $\mu$-a.e., which is enough to conclude.
\end{rem}

\printbibliography[heading=bibintoc, title={References}]

\end{document}